\documentclass[11pt]{article}
\usepackage[english]{babel}
\usepackage[utf8]{inputenc} 
\usepackage[T1]{fontenc} 
\usepackage{lmodern} 
\usepackage{amsmath}
\usepackage{amssymb}
\usepackage{amsthm}
\usepackage{hyperref}
\usepackage{graphicx}
\usepackage{titlesec}
\usepackage[font=footnotesize]{caption}
\usepackage[margin=3cm]{geometry}
\usepackage{enumitem}

\titleformat{\subsection}[runin]{\normalfont\itshape}{\thesubsection\hspace{8pt}}{3pt}{}[.] 

\newcommand\NN{\mathbb{N}} 
\newcommand\RR{\mathbb{R}} 
\newcommand\EE{\mathbb{E}} 
\newcommand\PP{\mathbb{P}} 
\newcommand\Law{\mathcal{L}} 
\newcommand\ind{\mathbf{1}} 
\newcommand\Z{\mathcal{Z}} 
\newcommand\E{\mathcal{E}}
\newcommand\ba{\mathbf{a}}
\newcommand\bb{\mathbf{b}}
\newcommand\I{\mathcal{I}}
\DeclareMathOperator{\var}{Var} 
\DeclareMathOperator{\supp}{supp} 
\newcommand\equald{\overset{\text{d}}{=}}

\newtheorem{theorem}{Theorem}
\newtheorem{lemma}[theorem]{Lemma}
\newtheorem{proposition}[theorem]{Proposition}
\newtheorem{definition}[theorem]{Definition}
\theoremstyle{definition}
\newtheorem{remark}[theorem]{Remark}

\begin{document}
	
	\title{Convergence and stationary distribution of Elo rating systems}
	
	\author{Roberto Cortez\footnote{Universidad Andres Bello, Departamento de Matemáticas, Sazié 2212, sexto piso, Santiago, Chile. E-mail: \texttt{roberto.cortez.m@unab.cl}.}
		\, and
		Hagop Tossounian\footnote{Universidad de Concepci\'{o}n, Departamento de Matem\'{a}tica, Avenida Esteban Iturra s/n, Barrio Universitario, Casilla 160-C, Concepci\'{o}n, Chile. E-mail: \texttt{htossounian@udec.cl}}}

	\maketitle
	
	\begin{abstract}
		The Elo rating system is a popular and widely adopted method for measuring the relative skill levels of players or teams in various sports and competitions. It assigns players numerical ratings and dynamically updates them based on game results and a model parameter $K$, which determines the sensitivity of rating changes. Assuming random games, this leads to a Markov chain for the evolution of the ratings of the $N$ players in the league. Despite its widespread use, little is known about the long-term behavior of this process. Aiming to fill this gap, in this article we prove that the process converges to its unique equilibrium distribution at an exponential rate in the 2-Wasserstein distance and almost surely. Moreover, we show important properties of the stationary distribution, such as the finiteness of an exponential moment, full support, and convergence to the players' true skills as $K$ decreases, at a rate of $\sqrt{K}$. We also provide Monte Carlo simulations that illustrate some of these properties and offer new insights.
	\end{abstract}
	

	\section{Introduction}

	\subsection{Elo rating}
	
	The Elo rating system, developed by physicist Arpad Elo in 1959 \cite{elo1978}, is a method for rating players or teams in sports, gaming and other competitive settings. Originally intended for chess, it was adopted by the World Chess Federation (FIDE) in 1970. Today, the Elo rating system and its variants are widely used in many competitions, including Go, football (soccer), Scrabble, and esports like League of Legends and Counter-Strike. More recently, the system has been applied to the evaluation of large language models (LLMs) as a method for comparing the performance of different LLMs based on their ability to generate high-quality responses \cite{boubdir-etal-2023,gonzalezbustamante2024}.
	
	The system assigns players a numerical rating based on their game results. Typically, when two players compete, the winner gains points while the defeated player loses them, with the number of points exchanged depending on the difference between their ratings. A player defeating a higher-rated opponent gains more points, while winning against a lower-rated opponent yields less points. The system is dynamic, updating ratings after each game to reflect a player's performance and relative skill.
	
	Specifically, after a game between two distinct players $i$ and $j$, their respective ratings $x^i, x^j \in \RR$ are updated according to the rule
	\begin{equation}
		\label{eq:EloRule}
		\begin{split}
			x^i & \leftarrow x^i + K \cdot \{S^{ij}-b(x^i-x^j)\} \\
			x^j & \leftarrow x^j - K \cdot \{S^{ij}-b(x^i-x^j)\},
		\end{split}
	\end{equation}
	where:
	\begin{itemize}
		\item $S^{ij}$ is the \emph{score}, a random number representing the outcome of this particular game from the perspective of $i$, whereas $S^{ji} = -S^{ij}$ is the score of $j$. When the only possible outcomes are win or lose, then $S^{ij}=1$ if player $i$ wins, and $S^{ij}=-1$ if player $i$ loses. If a tie can also occur, it corresponds to $S^{ij}=0$. More general outcomes can be considered by allowing $S^{ij}$ to take any value in $[-1,1]$.
		
		\item $b(x^i-x^j)$ is the algorithm's prediction of the expected score of the match. Here $b: \RR \to (-1,1)$ is a given increasing odd function. A popular choice is the \emph{logistic}: for some scaling constant $c>0$,
		\begin{equation}
			\label{eq:logistic}
			b(u)
			= \tanh(cu).
		\end{equation}
		
		\item $K>0$ is a constant called the \emph{$K$-factor}, determining the sensitivity of rating changes. The theoretical maximum number of points a player can win or lose after a single match is $2K$.
		
	\end{itemize}
	
	Notice that, since $S^{ji} = -S^{ij}$ and $b(\cdot)$ is odd, the rule \eqref{eq:EloRule} is symmetric with respect to the indices $i$ and $j$. This rule is then applied iteratively, sampling distinct players $i, j \in \{1,\ldots,N\}$ uniformly at random for each match, where $N$ is the total number of players in the league. This gives rise to a Markov chain $(X_t)_{t \in \NN} = (X_t^1,\ldots,X_t^N)_{t\in\NN}$ on $\RR^N$, which we call the \emph{Elo process}, see Definition \ref{def:EloProcess} below. Notice that \eqref{eq:EloRule} generates a zero-sum model, that is, the sum of the ratings of the $N$ players is conserved.
	
	We remark that a different convention for the score is often preferred: either $S^{ij}=1$ or $S^{ij}=0$, representing a win or loss for player $i$, respectively. The function $b$ then takes values in $(0,1)$, and $b(x^i-x^j)$ is interpreted as the algorithm's predicted probability that player $i$ wins. However, since this is mathematically equivalent to an instance of the general case $S^{ij} \in [-1,1]$, which better reflects the symmetry of the model, in this article we adopt the latter convention, as in \cite{during-torregrossa-wolfram2019,during-fisher-wolfram2022,jabin-junca2015}.
	
	The main goal of the algorithm is to compute ratings that hopefully reflect the skill of each player. To make this more precise, we need some modeling assumptions. Specifically, for each $i \in \{1,\ldots,N\}$, one assumes that there is some fixed number $\rho^i \in \RR$ measuring the \emph{true skill} or \emph{intrinsic strength} of player $i$. Moreover, in order to relate $\rho^i$ to the player's actual performance, it is assumed that the score $S^{ij}$ of a game between $i\neq j$ satisfies
	\begin{equation}
		\label{eq:ESij}
		\EE[S^{ij}]
		= b(\rho^i - \rho^j).
	\end{equation}
	That is, the expected score of player $i$ against $j$ depends only on the difference of skills $\rho^i - \rho^j$ through the function $b(\cdot)$. In practice, the true skills vector $\rho = (\rho^1,\ldots,\rho^N)$ is not known, and one hopes the collection of ratings $X_t = (X_t^1,\ldots,X_t^N)$ to be a reasonable estimate when time $t$ is large.

	\subsection{Relevant literature, main results, and methodology}
	
	Despite its widespread use, there are only a handful of works dedicated to the rigorous study of the mathematical features of Elo rating systems, all of them relatively new. Many important questions have remained largely open, most notably the nature of the large-time convergence of the Elo process, and the properties of its equilibrium distribution $\pi$. To the best of our knowledge, the first mathematical developments are carried out in 2015 by Avdeev \cite{avdeev2015,avdeev2015b}, for the case $N=2$. The first study for general $N$ is Aldous's 2017 work \cite{aldous2017,aldous2017b}, where a proof of existence and convergence in distribution to $\pi$ is provided. Quantitative convergence results for some variants of the model are proven later in \cite{junca2021} by Junca and in the recent article \cite{oleskerTaylor-Zanetti2024} by Olesker-Taylor and Zanetti. These variants of the model introduce simplifying assumptions such as strong contractivity or compactness of the state space. Some numerical analysis is performed in \cite{krifa-spinelli-junca2021,manCastillo-junca2024,manCastillo-junca2024b}.
	
	It is worth mentioning the work by Jabin and Junca \cite{jabin-junca2015}, where the authors develop a mean-field approach to the model by considering its scaling limit when $N$ is large and $K$ small, leading to partial differential equations modeling the time evolution of the distribution of ratings. Moreover, in \cite{during-fisher-wolfram2022,during-torregrossa-wolfram2019}, the authors consider extensions of these mean-field equations that include \emph{variable strengths}---that is, $\rho^i$ and $\rho^j$ may change randomly after a game. This allows one to study learning effects, performance variability, and related phenomena. In the setting of the present article (without the mean-field assumption), this generalization would make the model significantly more difficult to analyze. Therefore, we consider only the already challenging case of constant $\rho$, in the hope that our work will serve as a foundation for the analysis of more complex and realistic variants of the model.
	
	The main goal of the present article is to provide stronger large-time convergence results for the Elo process, valid for fixed $N \in \NN$ and $K>0$ (no scaling limit) and without strong simplifying assumptions. Moreover, we will prove relevant properties of the equilibrium distribution $\pi$, such as finiteness of an exponential moment, behavior as $K\to 0$, and the extent of its support.
	
	We now present our main results and discuss their relevance. Let us introduce the following notation. We denote $\Vert \cdot \Vert$ and $|\cdot|_1$ the Euclidean (or 2-norm) and 1-norm on $\RR^N$, respectively. For $p\geq 1$, let $W_p(\cdot,\cdot)$ be the $p$-Wasserstein distance for probability measures on $\RR^N$ with respect to $\Vert \cdot \Vert$. Let $\Law(\cdot)$ denote the law of a random element. Finally, we let $\Z_N$ be the zero-sum subspace of $\RR^N$:
	\[
	\Z_N
	:= \left\{ x\in \RR^N : \textstyle\sum_{i=1}^N x^i = 0 \right\}.
	\]
	In all that follows, we assume conditions \ref{ass:first}-\ref{ass:last}, given below. Let $(X_t)_{t \in \NN}$ be the Elo process, specified in Definition \ref{def:EloProcess}, with random initial condition $X_0 \in \Z_N$. Our first main result is the following:
	
	\begin{theorem}[exponential convergence to the stationary distribution $\pi$]
		\label{thm:main:geometric_convergence}
		There exist constants $\theta_0>0$ and $C>0$, such that for any $\theta \in (0,\theta_0]$, there exists $\kappa \in (0,1)$ depending on $\theta$, such that
		\[
		W_2^2(\Law(X_t), \pi) \leq C  (1-\kappa)^t \EE\left[ (1+\Vert X_0 \Vert^2)  e^{\theta \Vert X_0 \Vert}  \right].
		\]
		In particular, this yields exponential convergence to $\pi$ at rate $1-\kappa$ not depending on $\Law(X_0)$, in the class of initial distributions satisfying $\EE[\Vert X_0 \Vert^2 e^{\theta_0 \Vert X_0 \Vert}] < \infty$.
	\end{theorem}
	
	The proof, provided at the end of Section \ref{sec:convergence_stationary}, relies on the well-known method of establishing a Foster--Lyapunov drift condition (see our key Lemmas \ref{lem:Ex<-delta} and \ref{lem:Foster-Lyapunov-drift}). We also make heavy use of an explicit estimate for a specific coupling between two realizations of the Elo process, called the \emph{natural coupling} (see Definition \ref{def:natural_coupling} and Lemma \ref{lem:coupling_non_increasing}), from which one can deduce contraction on compact sets; see Lemma \ref{lem:contraction_natural_coupling_compact}. Once these two properties are proven, one can apply existing results for general Markov chains to obtain exponential convergence to $\pi$ similar to Theorem \ref{thm:main:geometric_convergence}; see, for instance, \cite{qin-hobert2022}. Nevertheless, since in our setting obtaining this convergence requires little additional work---mainly establishing the contraction estimate of Lemma \ref{lem:contraction_estimate}---for the sake of completeness and the reader's convenience, we provide a full proof.
	
	Note that the bound provided by Theorem \ref{thm:main:geometric_convergence} depends exponentially on $X_0$, which might seem suboptimal. However, this is actually unavoidable: since the jumps are bounded, if the process is far from the origin, it takes a linear amount of time to come back to a compact set; this makes it impossible to obtain an upper bound like $\tilde{C} (1-\kappa)^t$ for $\tilde{C}$ depending subexponentially on $X_0$. This is in the same vein as general results on exponential convergence for Markov chains, where the constant may depend (possibly exponentially) on the initial state of the process; see, for instance, \cite{down-meyne-tweedie1995,qin-hobert2022}.
	
	In the course of proving Theorem \ref{thm:main:geometric_convergence}, we will obtain several relevant results. We establish the boundedness of an exponential moment of both the Elo process and its stationary distribution $\pi$; see Theorems \ref{thm:unif_moments} and \ref{thm:stationary}. This may explain why, in practice, Elo ratings appear to be bounded, even though, mathematically, they are not; see Lemma \ref{lem:elo^t}. This may also support the validity of contractive or compact variants of the model \cite{junca2021,oleskerTaylor-Zanetti2024}. In Theorem \ref{thm:normX-Yto0}, we use the explicit estimate provided in Lemma \ref{lem:coupling_non_increasing} to show that, for the natural coupling, the distance between the two processes decreases to $0$ a.s. We also prove convergence to $\pi$ in $W_p$, which seems to be novel in this setting; see \cite{oleskerTaylor-Zanetti2024} for a quantitative result in $W_2$ for a compact variant. See the details in Theorem \ref{thm:stationary}, which, for the sake of completeness, also provides a proof of Aldous's result about the existence of $\pi$ and weak convergence to it.
	
	Even though the Elo algorithm attempts to estimate the vector $\rho$ of true skills, the ratings themselves are \emph{biased}: for $X\sim \pi$, typically $\EE[X] \neq \rho$. In the literature, this seems to be a well-known fact, supported by numerical experiments \cite{manCastillo-junca2024,manCastillo-junca2024b}; we provide further numerical evidence in Section \ref{sec:numerics:expected_rating} (in contrast, see Proposition \ref{prop:EbXiXj} for the estimation of $b(\rho^i-\rho^j)$). Consequently, an important question is whether $\pi$ converges to the true skills $\rho$ as the $K$-factor decreases. This is the content of our second main result:
	
	\begin{theorem}[convergence of $\pi$ to the true skills]
		\label{thm:main:stationary_sqrtK}
		There exists a constant $C>0$ depending (increasingly) only on $\max_i |\rho^i|$, such that, for all $K$ sufficiently small, we have for $X\sim \pi$:
		\[
		\EE\left[\frac{1}{N} | X - \rho |_1 \right]
		\leq C \sqrt{K}.
		\]
	\end{theorem}
	
	Thus, $\pi$ converges to $\delta_\rho$ (the Dirac mass at $\rho$) as $K\to 0$ in $W_1$; see the precise statement in Theorem \ref{thm:stationary_sqrtK}. This also controls the bias $\EE[X]-\rho$ and the mean absolute deviation $\EE[|X-\EE[X]|_1]$. The convergence is of order $\sqrt{K}$, which was also observed numerically in \cite{manCastillo-junca2024,manCastillo-junca2024b} for the standard deviation of $X$ in the case of two players; see \cite{oleskerTaylor-Zanetti2024} for a bound of the same order in a compact setting. Moreover, in Section \ref{sec:numerics:expectation_K} we provide numerical results that agree with our estimate, which suggest that the order $\sqrt{K}$ is sharp for small $K$ but false away from $0$. We remark that our estimate is uniform in $N$, provided that $\max_i |\rho^i|$ stays bounded.
	
	Finally, we also prove that $\pi$ has full support on $\Z_N$, which seems to be completely novel as well; see Theorem \ref{thm:support}. Nevertheless, some numerical and mathematical analyses pointed to (at least) an unbounded support \cite{krifa-spinelli-junca2021}.

	\subsection{Assumptions and definitions}
	
	We now provide a precise mathematical definition of the Elo process $(X_t)_{t \in \NN}$. We will work under the following assumptions throughout this article:
	
	\begin{enumerate}[label=(A\arabic*)]
		\item \label{ass:first}
		\label{ass:ZN}
		The vectors of random initial ratings $X_0 = (X_0^1,\ldots,X_0^N)$ and true skills $\rho = (\rho^1, \ldots, \rho^N)$ belong to the subspace of zero-sum $\Z_N$.
		
		\item \label{ass:b}
		The function $b:\RR \to (-1,1)$ is odd, strictly increasing and $L$-Lipschitz continuous. Moreover, it's inverse is Lipschitz on compact sets; that is, if we define
		\begin{equation}
			\label{eq:ell}
			\ell_M
			:= \inf_{-M \leq v < u \leq M} \frac{b(u)-b(v)}{u-v},
		\end{equation}
		then $\ell_M>0$ for all $M>0$. Note that $\ell_M$ is non-increasing as a function of $M$, and $\ell_M \to 0$ when $M\to\infty$, because $b(\cdot)$ is bounded.
		
		\item \label{ass:Sij}
		For every distinct $i,j \in \{1, \ldots, N\}$, denote $S^{ij}$ the random variable on the interval $[-1,1]$ corresponding to the score of a match between players $i$ and $j$. We assume \eqref{eq:ESij}, that is, $\EE[S^{ij}] = b(\rho^i - \rho^j)$. We also assume that $S^{ji} = - S^{ij}$, which is of course compatible with $b(\cdot)$ being odd.
		
		\item \label{ass:KL}
		The $K$-factor and Lipschitz constant $L$ satisfy $KL<1$.
		
		\item \label{ass:suppS}
		$\supp(S^{ij}) \supseteq \{-1,1\}$ for every distinct $i,j$. That is, $\PP(S^{ij} < -1+\epsilon) > 0$ and $\PP(S^{ij} > 1-\epsilon) > 0$ for all $\epsilon>0$.
		
		\label{ass:last}
	\end{enumerate}
	
	Some remarks about these assumptions:
	
	\begin{itemize}
		\item Assumption \ref{ass:ZN} is of course very natural, since the system preserves the sum of the ratings. It implies $X_t \in \Z_N$ for all $t\in \NN$.
		
		\item In assumption \ref{ass:b}, the condition $\ell_M > 0$ for all $M>0$ is quite mild. For instance, it is satisfied by any $C^1$ function $b(\cdot)$ with $b' > 0$, and by any odd function $b(\cdot)$ concave on $[0,\infty)$. In the latter case, which includes the logistic function \eqref{eq:logistic} and many others, we have $\ell_M = b'(M)$.
		
		\item
		The symmetry conditions of $b(\cdot)$ being odd in \ref{ass:b} and $S^{ji} = - S^{ij}$ in \ref{ass:Sij} do not entail a loss of generality: if the scores $\tilde{S}^{ij}$ and function $\tilde{b}(\cdot)$ do not satisfy these conditions, then one can work with $S^{ij} := B \tilde{S}^{ij} - (1-B)\tilde{S}^{ji}$, where $B \sim \text{Bernoulli}(1/2)$, independent of everything else, with corresponding odd function $b(u) = \frac{1}{2}(\tilde{b}(u)-\tilde{b}(-u))$. One reason for dropping these assumptions would be to include \emph{home-field advantage} in the model: the expected score of $i$ as a home team, i.e.\ $\EE[\tilde{S}^{ij}]$, is larger than the expected score of $i$ as an away team, i.e.\ $\EE[-\tilde{S}^{ji}]$. When the score takes values in $\{-1,1\}$, it is customary to achieve this by working with $\tilde{b}(\cdot) = b(\cdot + c)$ for some constant $c>0$, see for instance \cite{langville-Meyer2012,sismanis2010}.
		
		\item The condition $KL<1$ in \ref{ass:KL} is ubiquitous in the literature \cite{aldous2017,avdeev2015,avdeev2015b,junca2021,krifa-spinelli-junca2021,oleskerTaylor-Zanetti2024}. It means that $K$, the scale of the changes in rating after each game, is smaller than $1/L$, the scale of the ratings themselves. In practice, one works with $K \ll 1/L$, so it is not a restrictive assumption. For instance, typical values of $2K$ used in chess range from 10 to 40, whereas $1/L \sim 1000$.
		
		\item
		Assumption \ref{ass:suppS} implies that $S^{ij}$ can attain values close to $\pm 1$ with positive probability. It allows the rating of any given player to become unbounded. It will be required to prove that $\supp(\pi) = \Z_N$. This condition is easy to satisfy: if, for example, $\supp(S^{ij}) = [-c,c]$ for some $c<1$, we can simply work with the re-scaled random variable $S^{ij}/c$ instead.
	\end{itemize}

	\begin{definition}
		\label{def:EloProcess}
		Assume \ref{ass:first}-\ref{ass:last}. The \emph{Elo process} with parameters $\rho \in \Z_N$, $b(\cdot)$, $(\Law(S^{ij}))_{i\neq j}$ and $K>0$, is the Markov chain $(X_t)_{t\in\NN}$ on $\Z_N \subseteq \RR^N$, starting from a given random vector $X_0 \in \Z_N$, with the following dynamics: at step $t$,
		\begin{enumerate}
			\item select an ordered pair $(I_t, J_t) \in \{1,\ldots,N\}^2$ of two distinct players, uniformly at random;
			
			\item given the values $i=I_t$, $j=J_t$, sample a score random variable $S_t^{ij}$ which is an independent copy of $S^{ij}$.
			
			\item update the coordinates $i$ and $j$ of $X_t$:
			\begin{align*}
				X_{t+1}^i
				&= X_t^i + K \{S^{ij}_t - b(X_t^i - X_t^j)\} \\
				X_{t+1}^j
				&= X_t^j - K \{S^{ij}_t - b(X_t^i - X_t^j)\}.
			\end{align*}
		\end{enumerate}
	\end{definition} 
	
	\begin{remark}
		\label{rmk:xIJ}
		Equivalently, when the process is at state $x \in \Z_N$, it is updated to
		\begin{equation}
			\label{eq:xIJ}
			x + K \{S^{IJ} - b(x^I - x^J)\} (e_I - e_J),
		\end{equation}
		where $e_i$ is the $i$-th canonical vector of $\RR^N$, the pair $(I,J)$ is chosen uniformly at random among all distinct pairs, and, conditionally on $I=i, J=j$, the variable $S^{IJ}$ is an independent copy of $S^{ij}$.    
	\end{remark}
	
	There is an obvious coupling between two Elo processes, given by the following Definition. This coupling was already used by Aldous \cite{aldous2017b}, and also by Olesker-Taylor and Zanetti \cite{oleskerTaylor-Zanetti2024}. We will take advantage of it throughout the present work.
	
	\begin{definition}
		\label{def:natural_coupling}
		Given two arbitrary random initial conditions $X_0, Y_0 \in \Z_N$, the \emph{natural coupling} $(X_t,Y_t)_{t\in \NN}$ is the process obtained by using exactly the same randomness at each step for both $(X_t)_{t\in \NN}$ and $(Y_t)_{t\in \NN}$. That is, at each step $t\in\NN$, both processes use the same pair of players $i=I_t$, $j=J_t$ and the same realization of the score random variable $S_t^{ij}$.
	\end{definition}

	\subsection{Plan of the paper}
	
	The remainder of this article is structured as follows. In Section \ref{sec:Elo_process}, we study some mathematical properties of the Elo process $(X_t)_{t\in\NN}$, especially for large $t$. Section \ref{sec:convergence_stationary} is devoted to the convergence of the process to its stationary distribution $\pi$. We study some of the properties of $\pi$ in Section \ref{sec:properties_stationary}. In Section \ref{sec:numerics}, we present some Monte Carlo simulations that illustrate relevant features of $\pi$. Finally, in Section \ref{sec:conclusion}, we provide some concluding remarks and mention possible lines of future research.
	
	For the reader's convenience, let us mention our overall strategy and the dependence structure among the main results:
	\begin{itemize}
		\item Our main tools are an explicit estimate for the natural coupling (Lemma \ref{lem:coupling_non_increasing}) and a Foster--Lyapunov drift condition (Lemmas \ref{lem:Ex<-delta} and \ref{lem:Foster-Lyapunov-drift}).
		
		\item These tools lead to a uniform bound of some exponential moments (Theorem \ref{thm:unif_moments}), and then a.s.\ convergence to $0$ for the natural coupling (Theorem \ref{thm:normX-Yto0}).
		
		\item This leads to a proof of the existence of $\pi$, convergence to it, and finiteness of an exponential moment (Theorem \ref{thm:stationary}).
		
		\item The previous results allow us to derive a contraction estimate (Lemma \ref{lem:contraction_estimate}), which then leads to exponential convergence (Theorem \ref{thm:main:geometric_convergence}).
		
		\item The analysis of the properties of $\pi$ (convergence as $K\to 0$ in Theorem \ref{thm:stationary_sqrtK} and full support in Theorem \ref{thm:support}) is largely independent of the rest.
		
	\end{itemize}

	\section{The Elo process}
	\label{sec:Elo_process}
	
	In order to state the next result, let us fix some notation. Given the indices of two distinct players $i,j \in \{1,\ldots,N\}$ and a (non-random) number $s \in [-1,1]$, denote $\alpha = (i,j,s)$. Given a fixed such $\alpha$ and some $x \in \RR^N$, denote $\E_\alpha(x)$ the vector obtained after a single step of the Elo algorithm corresponding to the information in $\alpha$. That is,
	\begin{equation}
		\label{eq:E_alpha}
		\E_\alpha(x)
		= x + K\{s-b(x^i-x^j)\}(e_i - e_j).
	\end{equation}
	The following is a simple yet extremely useful estimate. A similar computation is performed in \cite{oleskerTaylor-Zanetti2024}.
	
	\begin{lemma}
		\label{lem:coupling}
		For any vectors $x,y \in \RR^N$, any distinct $i,j \in \{1,\ldots,N\}$ and any $s \in [-1,1]$, for $\alpha=(i,j,s)$ we have
		\begin{align*}
			& \left\Vert \E_\alpha(x) - \E_\alpha(y) \right\Vert^2 - \left\Vert x - y \right\Vert^2 \\
			&\leq {} - 2K (1-KL) \left| x^i-x^j - y^i+y^j \right| \left| b(x^i-x^j)-b(y^i-y^j) \right|.
		\end{align*}
	\end{lemma}
	
	\begin{proof}
		Calling $\E_\alpha^i(x)$ the $i$-th entry of the vector $\E_\alpha(x) \in \RR^N$, clearly we have:
		\begin{align*}
			\E_\alpha^i(x) - \E_\alpha^i(y)
			&= x^i + K\{s-b(x^i-x^j)\} - y^i - K\{s-b(y^i-y^j)\} \\
			&= (x^i - y^i) - K\{b(x^i-x^j) -b(y^i-y^j)\}.
		\end{align*}
		To shorten the notation, call $\Delta_{ij} = b(x^i-x^j) -b(y^i-y^j)$. Thus:
		\[
		(\E_\alpha^i(x) - \E_\alpha^i(y))^2
		= (x^i - y^i)^2
		- 2K(x^i - y^i)\Delta_{ij}
		+ K^2 \Delta_{ij}^2.
		\]
		A similar computation for $j$ yields the analogous identity, but with a plus sign in front of the cross term:
		\[
		(\E_\alpha^j(x) - \E_\alpha^j(y))^2
		= (x^j - y^j)^2
		+ 2K(x^j - y^j)\Delta_{ij}
		+ K^2 \Delta_{ij}^2.
		\]
		Notice that the vectors $\E_\alpha(x) - \E_\alpha(y)$ and $x-y$ differ only in their $i$-th and $j$-th entries. Thus, from the last two identities, we obtain
		\begin{align*}
			&\left\Vert \E_\alpha(x) - \E_\alpha(y) \right\Vert^2
			- \left\Vert x - y \right\Vert^2 \\
			&= (\E_\alpha^i(x) - \E_\alpha^i(y))^2
			+ (\E_\alpha^j(x) - \E_\alpha^j(y))^2
			- (x^i - y^i)^2
			- (x^j - y^j)^2 \\
			&= - 2K(x^i - x^j - y^i + y^j) \Delta_{ij}
			+ 2K^2 \Delta_{ij}^2 \\
			&\leq - 2K(x^i - x^j - y^i + y^j) \Delta_{ij}
			+ 2K^2 L |x^i - x^j - y^i + y^j | |\Delta_{ij}|,
		\end{align*}
		where in the last step we used that the function $b$ is $L$-Lipschitz. Moreover, since $b$ is increasing, $\Delta_{ij}$ has the same sign as $x^i - x^j - y^i + y^j$, and then we can insert absolute values to the factors in the first term of the last line. The desired estimate follows.
	\end{proof}

	As an immediate consequence of Lemma \ref{lem:coupling}, we have the following extremely useful result:
	
	\begin{lemma}
		\label{lem:coupling_non_increasing}
		The natural coupling $(X_t,Y_t)_{t\in \NN}$ almost surely satisfies for all $t\in\NN$,
		\begin{equation}
			\label{eq:normX-Ynon_increasing}
			\begin{split}
				& \left\Vert X_{t+1} - Y_{t+1} \right\Vert^2 - \left\Vert X_t - Y_t \right\Vert^2 \\
				&\leq {} - 2K (1-KL) \left| X_t^{I_t}-X_t^{J_t} - Y_t^{I_t}+Y_t^{J_t} \right| \left| b(X_t^{I_t}-X_t^{J_t})-b(Y_t^{I_t}-Y_t^{J_t}) \right|,
			\end{split}
		\end{equation}
		where $(I_t,J_t)$ is the (random) pair of players involved in the $(t+1)$-th match. Consequently, since we are assuming $KL < 1$, we have that $\Vert X_t - Y_t \Vert$ is almost surely non-increasing.
	\end{lemma}

	The following estimate is essential. It will allow us to prove a Foster--Lyapunov drift condition for the Elo process. In what follows, for any $x\in \Z_N$, we denote $\EE_x$ the expectation conditional on $X_0 = x$. Also note that, since the ratings change by at most $2K$, the norm of the jumps of the process are a.s.\ bounded by $\sqrt{8}K$; we will use this fact frequently.
	
	\begin{lemma}
		\label{lem:Ex<-delta}
		There exist constants $R_0>0$ and $\delta>0$, such that for all $x \in \Z_N$ with $\Vert x \Vert > R_0$,
		\[
		\EE_x[X_1 - x] \cdot x \leq - \delta \Vert x \Vert.
		\]
	\end{lemma}
	
	\begin{proof}
		Given $x \in \Z_N$, let $i_* \neq j_*$ be given by $i_* = \operatorname{argmax}_i(x^i - \rho^i)$ and $j_* = \operatorname{argmin}_j (x^j -\rho^j)$. Both depend on $x$.
		Since $\sum_i x^i = 0$, taking $\Vert x\Vert$ large enough, we can force $x^{i_*}-x^{j_*}$ to be as large as we want, or equivalently, $b(x^{i_*} - x^{j_*})$ as close to 1 as we want. More specifically: defining
		\[
		\eta
		= 1 - \max_{i,j} b(\rho^i-\rho^j)
		>0,
		\]
		there exists a constant $A>0$ depending only on $b$ and the vector $\rho$, such that
		\begin{equation}
			\label{eq:brhoij-bxij}
			|b(\rho^{i_*} - \rho^{j_*}) - b(x^{i_*}-x^{j_*})|
			\geq \frac{\eta}{2},
			\qquad \text{when $\Vert x \Vert > A$}.    
		\end{equation}
		Also, since $\sum_i (x^i-\rho^i) = 0$, then $x^{i_*} - \rho^{i_*}$ and $\rho^{j_*}-x^{j_*}$ are non-negative, and one of them attains the maximum of $|x^i - \rho^i|$. Consequently,
		\begin{equation}
			\label{eq:xij-rhoij_norm_infty}
			x^{i_*}-x^{j_*} - (\rho^{i_*} - \rho^{j_*})
			\geq \Vert x-\rho\Vert_\infty \geq N^{-1/2} \Vert x-\rho\Vert.
		\end{equation}
		Call $\Delta = X_1 -x$. From \eqref{eq:xIJ}, we have
		\begin{align*}
			\EE_x [\Delta] \cdot (x-\rho) 
			&= \frac{K}{N(N-1)} \sum_{i,j=1}^N
			\EE [ S^{ij}-b(x^i-x^j)](e_i - e_j) \cdot (x-\rho)  \\
			&= \frac{K}{N(N-1)} \sum_{i,j=1}^N (b(\rho^i -\rho^j)-b(x^i - x^j))(x^i -\rho^i- x^j + \rho^j),
		\end{align*}
		where we have used the equality $\EE[S^{ij}]=b(\rho^i - \rho^j)$. Since $b(\cdot)$ is increasing, all the terms are non-positive. Discarding all but one of the terms, we get
		\begin{align*}
			\EE_x [\Delta ] \cdot (x-\rho)
			& \leq -\frac{K}{N(N-1)} \vert b(x^{i_*}-x^{j_*})- b(\rho^{i_*}-\rho^{j_*})\vert \vert x^{i_*}-x^{j_*}- (\rho^{i_*}-\rho^{j_*}) \vert \\
			& \leq -\frac{\eta}{2}\frac{K}{N(N-1)} \vert x^{i_*}-x^{j_*}- (\rho^{i_*}-\rho^{j_*})\vert \\
			& \leq -\frac{\eta K}{2 N^{3/2}(N-1)}\Vert x-\rho\Vert=: -2\delta \Vert x-\rho\Vert,
		\end{align*}
		provided $\Vert x\Vert \geq A$, where we have used \eqref{eq:brhoij-bxij} and \eqref{eq:xij-rhoij_norm_infty}. Since $\Vert \Delta \Vert \leq \sqrt{8}K$ a.s., we have
		\begin{align*}
			\EE_x[\Delta] \cdot \frac{x}{\Vert x \Vert}
			&= \EE_x[\Delta] \cdot \frac{(x-\rho)}{\Vert x \Vert} + \EE_x[\Delta] \cdot \frac{\rho}{\Vert x \Vert} \\
			&\leq -2\delta \frac{\Vert x-\rho\Vert}{\Vert x\Vert} + \sqrt{8}K \frac{\Vert \rho \Vert}{\Vert x\Vert}
		\end{align*}
		whenever$\Vert x \Vert > A$. By taking $\Vert x\Vert > R_0$ for $R_0$ large enough, the above expression becomes $\leq -\delta$; for example, $R_0:=\max\{A, \Vert \rho \Vert (2 + \sqrt{8}K/\delta)\}$ suffices. The result is proven.
	\end{proof}

	Together with Lemma \ref{lem:Ex<-delta}, the following is our central estimate. It establishes a Foster--Lyapunov drift condition for the Elo process, allowing us to prove contraction and uniform exponential moments. In what follows, given $\theta >0$, we denote
	\[
	V_\theta(x) := e^{\theta \Vert x \Vert}.
	\]
	
	\begin{lemma}[Foster--Lyapunov drift condition]
		\label{lem:Foster-Lyapunov-drift}
		There exist a constant $\theta_0>0$ with the following property: for any $\theta\in(0,\theta_0]$, there exist $\lambda \in (0,1)$, $R_1>0$ and $c>0$ (depending on $\theta$), such that for any $x \in \Z_N$,
		\[
		\EE_x[V_\theta(X_1)]
		\leq \lambda V_\theta(x) + c \ind_{\Vert x \Vert \leq R_1}.
		\]
	\end{lemma}
	
	\begin{proof}
		The proof is standard for any Markov process with bounded jumps satisfying an inequality as the one in Lemma \ref{lem:Ex<-delta}. Call $\Delta = X_1 - x$. Thus:
		\[
		\Vert X_1 \Vert
		= (\Vert x \Vert^2 + \Vert \Delta \Vert^2 + 2 \Delta \cdot x)^{1/2}
		\leq \Vert x \Vert + \frac{\Vert \Delta \Vert^2}{2 \Vert x \Vert} + \Delta \cdot \frac{x}{\Vert x \Vert},
		\]
		where we have used Bernoulli's inequality. Recall that $\Vert \Delta \Vert^2 \leq 8K^2$ a.s. Then:
		\[
		\EE_x[V_\theta(X_1)]
		\leq e^{\theta \Vert x \Vert} e^{\frac{4 \theta K^2}{\Vert x \Vert}} \EE_x[e^{\theta \Delta \cdot \frac{x}{\Vert x \Vert}}].
		\]
		Recall Hoeffding's lemma: one has $\EE[e^{\theta Z}] \leq e^{\theta \EE[Z] + \frac{\theta^2(c_1-c_0)^2}{8}}$ for any random variable $Z$ such that $c_0 \leq Z \leq c_1$ a.s. Using this with $Z = \Delta \cdot \frac{x}{\Vert x \Vert} \in [-\sqrt{8}K, \sqrt{8}K]$, we obtain:
		\[
		\frac{\EE_x[V_\theta(X_1)]}{V_\theta(x)}
		\leq e^{\frac{4 \theta K^2}{\Vert x \Vert}} e^{\theta \EE_x[Z] + 4 \theta^2 K^2}
		\leq e^{\frac{4 \theta K^2}{\Vert x \Vert} -  \theta \delta + 4\theta^2 K^2}
		\]
		when $\Vert x \Vert > R_0$, where we have used Lemma \ref{lem:Ex<-delta}. The last expression can be made smaller than $1$ uniformly in $x$: set $\theta_0 = \frac{\delta}{16 K^2}$, so that for $\theta \in (0,\theta_0]$ and $\Vert x \Vert > 1/\theta$, we have
		\[
		\frac{4 \theta K^2}{\Vert x \Vert} -  \theta \delta + 4\theta^2 K^2
		\leq 8 \theta^2 K^2 -  \theta \delta
		\leq - \frac{\theta \delta}{2}.
		\]
		Thus, we have shown that $\EE[V_\theta(X_1)] \leq \lambda V_\theta(x)$ when $\Vert x \Vert > R_1 := \max(R_0,1/\theta)$, for $\lambda := e^{-\theta \delta /2}$. On the other hand, when $\Vert x \Vert \leq R_1$, we have $\EE_x[V_\theta(X_1)] \leq c := V_\theta(R_1+\sqrt{8}K)$.
	\end{proof}

	As a consequence of Lemma \ref{lem:Foster-Lyapunov-drift}, the next result shows that the system $(X_t)_{t\in \NN}$ has bounded exponential moment of some small order, uniformly in $t$.

	\begin{theorem}[uniform moments]
		\label{thm:unif_moments}
		Let $\phi(x) = \Vert x \Vert^p e^{\theta \Vert x \Vert}$ for some $p \in [0,\infty)$ and $\theta \in [0,\theta_0]$, where $\theta_0$ is the constant provided by Lemma \ref{lem:Foster-Lyapunov-drift}. If $\EE[\phi(X_0)] < \infty$, then
		\[
		\sup_{t \in \NN} \EE[\phi(X_t)]
		< \infty.
		\]
	\end{theorem}
	
	\begin{proof}
		First, consider the case $p=0$, thus $\phi(x)= e^{\theta \Vert x\Vert} = V_\theta(x)$. Denote $q_t = \EE[\phi(X_t)]$. From Lemma \ref{lem:Foster-Lyapunov-drift}, we obtain $q_{t+1} \leq \lambda q_t + c$. Iterating this yields
		\[
		q_t
		\leq \lambda^t q_0 + c \sum_{k=0}^{t-1} \lambda^k
		\leq \lambda^t q_0 + \frac{c}{1-\lambda}, 
		\]
		thus $\sup_t q_t < \infty$.
		
		Next, consider the case $\theta=0$, thus $\phi(x) = \Vert x \Vert^p$. Starting from $Y_0 \equiv 0$, define $(X_t, Y_t)_{t\in\NN}$ as the natural coupling described in Definition \ref{def:natural_coupling}. Since $KL \leq 1$, we know that $\Vert X_t - Y_t \Vert$ is a.s.\ non-increasing (Lemma \ref{lem:coupling_non_increasing}), thus $\Vert X_t - Y_t \Vert \leq \Vert X_0 - Y_0 \Vert = \Vert X_0 \Vert$. Therefore,
		\begin{align*}
			\Vert X_t \Vert^p
			&\leq C_p \Vert X_t - Y_t \Vert^p + C_p \Vert Y_t \Vert^p \\
			&\leq C_p \Vert X_0 \Vert^p + C_p \Vert Y_t \Vert^p.
		\end{align*}
		Note that $\sup_t \EE[\Vert Y_t \Vert^p] < \infty$, thanks to the first case. Consequently, $\sup_t \EE [\Vert X_t \Vert^p] < \infty$, as desired.
		
		Now, consider the case $p \in (0,1]$, $\theta \in (0,\theta_0]$. Given that $X_0 = x$, recall that $\Vert X_1 - x \Vert$ is a.s.\ bounded by $M:=\sqrt{8}K$. Using Lemma \ref{lem:Foster-Lyapunov-drift} again, we have:
		\begin{align*}
			\EE_x[ \phi(X_1)]
			& \leq (\Vert x \Vert + M)^p \EE_x[e^{\theta \Vert X_1 \Vert}] \\
			& \leq (\Vert x \Vert^p + M^p) (\lambda e^{\theta \Vert x \Vert} + c) \\
			& = \lambda \phi(x) + c \Vert x \Vert^p + \lambda M^p e^{\theta \Vert x \Vert} + c M^p.
		\end{align*}
		Thanks to the previous cases, we know that $\EE[\Vert X_t \Vert^p]$ and $\EE[e^{\theta \Vert X_t \Vert}]$ are finite, uniformly in $t$. Thus, denoting $q_t = \EE[\phi(X_t)]$, the previous inequality yields $q_{t+1} \leq \lambda q_t + \tilde{c}$. Arguing as in the first case, we deduce $\sup_t q_t < \infty$.
		
		Finally, the case $p\in(1,\infty)$, $\theta \in (0,\theta_0]$ can be deduced similarly by noting that
		\[
		\EE_x[ \phi(X_1)]
		\leq \left(\Vert x \Vert^p + M^p + C_p M \Vert x \Vert^{p-1} + C_p M^{p-1} \Vert x \Vert \right)
		\EE_x[e^{\theta \Vert X_1 \Vert}]
		\]
		and applying a recursive argument.
	\end{proof}

	The following elementary computation will be of use. Recall that $\Vert \cdot \Vert$ denotes the $2$-norm (Euclidean norm), whereas $|\cdot|_1$ is the $1$-norm.
	
	\begin{lemma}
		\label{lem:normx}
		For any vector $x \in \Z_N$, the following holds:
		\[
		\Vert x \Vert^2
		= \frac{1}{2N} \sum_{i,j=1}^N (x^i - x^j)^2
		\qquad \text{and} \qquad
		|x|_1
		\leq \frac{1}{N} \sum_{i,j=1}^N |x^i - x^j|.
		\]
	\end{lemma}
	
	\begin{proof}
		Indeed:
		\begin{align*}
			\sum_{i,j=1}^N (x^i - x^j)^2
			&= \sum_{i,j=1}^N \left[ (x^i)^2 + (x^j)^2 - 2x^i x^j \right] \\
			&= 2N \Vert x \Vert^2 - 2 \sum_{i,j=1}^N x^i x^j \\
			&= 2N \Vert x \Vert^2,
		\end{align*}
		which proves the first assertion. For the second one, without loss of generality, we can assume that $x^i \geq 0$ for $i=1,2, \dots, i_*$ and $x^i \leq 0$ otherwise. Since $\sum_i x^i = 0$, we obtain $\sum_{i\leq i_*} x^i = \sum_{i > i_*} (-x^i) = \frac{|x|_1}{2}$. Then:
		\[
		\sum_{i,j=1}^N | x^i - x^j |
		= 2 \sum_{i<j}^N | x^i - x^j |
		\geq 2 \sum_{i \leq i_* < j} (x^i - x^j),
		\]
		which equals $(N-i_*) |x|_1 + i_* |x|_1 = N |x|_1$.
	\end{proof}

	With the previous results, in particular the coupling described in Lemma \ref{lem:coupling_non_increasing} and the uniform moments of Theorem \ref{thm:unif_moments}, we are now ready to prove the following result:
	
	\begin{theorem}[almost-sure convergence]
		\label{thm:normX-Yto0}
		Let $(X_t,Y_t)_{t \in \NN}$ be the natural coupling described in Definition \ref{def:natural_coupling}, starting from any random vectors $X_0, Y_0 \in \Z_N$. Then $\Vert X_t - Y_t \Vert$ is non-increasing a.s., and
		\[
		\lim_{t\to\infty} \Vert X_t - Y_t \Vert = 0
		\quad \text{a.s.}
		\]
	\end{theorem}
	
	\begin{proof}
		We already know from Lemma \ref{lem:coupling_non_increasing} that $\Vert X_t - Y_t \Vert$ is a.s.\ non-increasing when $KL<1$. Call $H = \lim_t \Vert X_t - Y_t \Vert$; our goal is to prove $H \equiv 0$ a.s. From \eqref{eq:normX-Ynon_increasing}:
		\begin{align*}
			& 2K (1-KL) \left| X_t^{I_t}-X_t^{J_t} - Y_t^{I_t}+Y_t^{J_t} \right| \left| b(X_t^{I_t}-X_t^{J_t})-b(Y_t^{I_t}-Y_t^{J_t}) \right| \\
			&\leq \left\Vert X_t - Y_t \right\Vert^2 - \left\Vert X_{t+1} - Y_{t+1} \right\Vert^2,
		\end{align*}
		where $(I_t,J_t)$ is the (random) pair of players involved in the $(t+1)$-th match. Notice that both terms on the right-hand side converge to $H^2$ a.s., which forces the left-hand side to converge to 0; the main idea of the proof is to transfer this convergence to $\Vert X_t - Y_t \Vert$.
		
		Using that $|u-v| \geq |b(u)-b(v)|/L$, we thus get
		\[
		\lim_{t\to\infty} \left( b(X_t^{I_t}-X_t^{J_t})-b(Y_t^{I_t}-Y_t^{J_t}) \right)^2
		= 0
		\quad \text{a.s.}
		\]
		This implies that the expected value of the sequence converges to 0, because $b$ is bounded. Since $(I_t,J_t)$ is chosen uniformly over all possible pairs and independently of $(X_t,Y_t)$, we obtain
		\[
		\lim_{t\to\infty} \EE\left[ \sum_{i,j=1}^N \left( b(X_t^i-X_t^j)-b(Y_t^i-Y_t^j) \right)^2 \right]
		= 0.
		\]
		Observe that, for any $M>0$, we have $|b(u)-b(v)| \geq \ell_M |u-v|$ when $u,v \in [-M,M]$, where $\ell_M>0$ was defined in \eqref{eq:ell}. Thus, defining
		\[
		M_t := \max_{i,j} \max \{ |X_t^i-X_t^j|, |Y_t^i-Y_t^j| \},
		\]
		gives
		\[
		\lim_{t\to\infty} \EE\left[ (\ell_{M_t})^2 \sum_{i,j = 1}^N \left( X_t^i-X_t^j-Y_t^i+Y_t^j \right)^2 \right]
		= 0.
		\]
		Since $X_t - Y_t \in \Z_N$, the last summation equals $2N \Vert X_t - Y_t \Vert^2$, thanks to Lemma \ref{lem:normx}. Thus, $\lim_t \EE[(\ell_{M_t})^2 \Vert X_t - Y_t \Vert^2] = 0$. This implies that for some (non-random) subsequence, we have
		\[
		\lim_{t\to\infty} \ell_{M_t} \Vert X_t - Y_t \Vert
		= H \lim_{t\to\infty} \ell_{M_t} = 0
		\quad \text{a.s.}
		\]
		Let's argue by contradiction: suppose that $H>0$ with positive probability. Then, a.s.\ on the event $\{H>0\}$, we have $\ell_{M_t} \to 0$, or equivalently, $M_t \to \infty$, because $\ell_M$ decreases to $0$ as $M$ increases, thanks to Assumption \ref{ass:b}. Let $(Z_t)_{t\in \NN}$ be the natural coupling starting from $Z_0 \equiv 0$. Again using that $\Vert X_t - Z_t \Vert$ is non-increasing, we have $\Vert X_t \Vert \leq \Vert X_t - Z_t\Vert + \Vert Z_t \Vert \leq \Vert X_0 \Vert + \Vert Z_t \Vert$; the same holds for $Y_t$. Thus:
		\[
		M_t
		\leq \sum_{i,j=1}^N \{ |X_t^i| + |X_t^j| + |Y_t^i| + |Y_t^j| \}
		\leq C_N \{ \Vert X_t \Vert + \Vert Y_t \Vert \}
		\leq C_N \{ \Vert X_0 \Vert + \Vert Y_0 \Vert + 2 \Vert Z_t \Vert \}.
		\]
		This forces $\Vert Z_t \Vert \to \infty$ with positive probability, which implies $\EE[\Vert Z_t \Vert] \to \infty$ for this subsequence. But this contradicts the fact that $\sup_t \EE [ \Vert Z_t \Vert ] < \infty$, granted by Theorem \ref{thm:unif_moments}. Therefore, we must have $H\equiv 0$ a.s.
	\end{proof}

	\section{Convergence to the Elo stationary distribution}
	\label{sec:convergence_stationary}
	
	We now study the convergence to the equilibrium distribution of the Elo process. With the developments of the previous section, we are ready to prove the following result. For the sake of completeness, it provides a proof of the existence and weak convergence to $\pi$. Recall that, for $p\geq 1$, the $p$-Wasserstein distance between two probability measures $\mu$ and $\nu$ on $\RR^N$, is defined as
	\[
	W_p(\mu,\nu)
	= \inf \left( \EE[\Vert Y - Z \Vert^p ] \right)^{1/p},
	\]
	where the infimum is taken over all couplings of $\mu$ and $\nu$, that is, over all possible random pairs $(Y,Z)$ defined on a common probability space, such that $Y \sim \mu$ and $Z \sim \nu$.

	\begin{theorem}[stationary distribution $\pi$]
		\label{thm:stationary}
		There exists a unique stationary distribution $\pi$ for $(X_t)_{t \in \NN}$, which satisfies:
		\begin{enumerate}
			\item \label{thm:stationary:i}
			$\EE[\Vert X \Vert^p e^{\theta_0 \Vert X\Vert}] <\infty$ when $X\sim \pi$, for all $p\geq 0$, where $\theta_0>0$ is provided by Lemma \ref{lem:Foster-Lyapunov-drift}.
			
			\item \label{thm:stationary:ii}
			$\lim_t \Law(X_t) = \pi$ weakly.
			
			\item \label{thm:stationary:iii}
			If $\EE[\Vert X_0 \Vert^p] < \infty$ for some $p\geq 1$, then $\lim_t W_p(\Law(X_t), \pi) = 0$.
			
		\end{enumerate}
	\end{theorem}
	
	\begin{proof}
		Denote $\Theta : \mathcal{P}(\Z_N) \to \mathcal{P}(\Z_N)$ the mapping that associates with $\nu \in \mathcal{P}(\Z_N)$ the law of the vector obtained after performing one step of the Elo scheme, starting from a $\nu$-distributed random vector. That is, from \eqref{eq:xIJ}, we see that for any test function $\phi$,
		\begin{align*}
			&\int_{\Z_N} \phi(x) \Theta \nu(dx) \\
			&= \frac{1}{N(N-1)} \sum_{i \neq j} \int_{\Z_N} \int_{-1}^1 \phi\left(x + K\{s-b(x^i-x^j)\} (e_i - e_j) \right) \sigma^{ij}(ds) \nu(dx) \\
			&= \frac{1}{N(N-1)} \sum_{i \neq j} \int_{\Z_N} \int_{-1}^1 \phi\left(x - b(x^i-x^j) (e_i-e_j)  + K s (e_i-e_j) \right) \sigma^{ij}(ds) \nu(dx),
		\end{align*}
		where $\sigma^{ij} = \Law(S^{ij})$. Thanks to the convolutional structure of the last expression and the continuity of $b(\cdot)$, it is readily seen that $\Theta$ is continuous with respect to weak convergence.
		
		Take any $\nu \in \mathcal{P}(\Z_N)$ such that $\int \phi(x) \nu(dx) < \infty$ for $\phi(x)= \Vert x \Vert^p e^{\theta_0 \Vert x \Vert}$ (for instance, the Dirac mass at $(0,\ldots,0)$), and consider the Cesàro sums $\mu_t := \frac{1}{t} \sum_{k=0}^{t-1} \Theta^k \nu$. Then, $\sup_t \int \phi(x) \mu_t(dx) < \infty$, thanks to Theorem \ref{thm:unif_moments}. Thus, $(\mu_t)_{t \in \NN}$ is tight, which implies that it admits at least one weak limit. Since $\Theta$ is continuous, by the Krylov-Bogolyubov procedure (see for instance \cite[Theorem 1.5.8]{arnold1998}), every such limit must be invariant for $\Theta$. This proves the existence of $\pi$. Uniqueness is a direct consequence of Theorem \ref{thm:normX-Yto0}.
		
		Clearly, the finiteness of $\int \phi(x) \pi(dx)$ is inherited from the uniform bound for $\mu_t$, proving \ref{thm:stationary:i}. Point \ref{thm:stationary:ii} is again consequence of Theorem \ref{thm:normX-Yto0}: take $(X_t,Y_t)$ to be the natural coupling (Definition \ref{def:natural_coupling}) starting from $X_0$ and $Y_0 \sim \pi$, thus $Y_t \sim  \pi$ for all $t \in \NN$. Since $\Vert X_t - Y_t\Vert$ a.s.\ decreases to $0$, we deduce that $\Law(X_t) \to \pi$. Moreover, if $\EE[\Vert X_0 \Vert^p]<\infty$, then the dominated convergence theorem gives $\EE[\Vert X_t - Y_t \Vert^p] \to 0$, which proves \ref{thm:stationary:iii}.
	\end{proof}

	We now aim to establish exponential convergence to $\pi$. To this end, the following result provides a contraction estimate on compact sets for the natural coupling. It will be useful to prove our main contraction estimate (Lemma \ref{lem:contraction_estimate}). Given a coupling $(X_t,Y_t)$, we denote $\EE_{x,y}$ the expectation conditional on $(X_0,Y_0) = (x,y)$.
	
	\begin{lemma}
		\label{lem:contraction_natural_coupling_compact}
		Let $(X_t,Y_t)$ be the natural coupling. Let $\omega(x,y) := \frac{4K(1-KL)}{N} \ell_{2\max(\Vert x \Vert, \Vert y \Vert)} >0$. Then, for any $x,y \in \Z_N$, we have
		\[
		\EE_{x,y}[\Vert X_1 - Y_1 \Vert^2]
		\leq (1-\omega(x,y)) \Vert x-y \Vert^2.
		\]
	\end{lemma}
	
	\begin{proof}
		Starting from $(X_0,Y_0) = (x,y)$, Lemma \ref{lem:coupling_non_increasing} gives
		\[
		\Vert X_1 - Y_1 \Vert^2 - \Vert x-y \Vert^2
		\leq - C |x^I-x^J - y^I + y^J| \, |b(x^I-x^J) - b(y^I - y^J)|,
		\]
		for $C = 2K(1-KL) > 0$, where $(I,J)$ is chosen uniformly at random among all pairs of distinct indices. Thanks to Assumption \ref{ass:b}, we have
		\begin{align*}
			|b(x^I-x^J) - b(y^I - y^J)|
			&\geq \ell_{\max(|x^I-x^J|,|y^I-y^J|)} |x^I-x^J - y^I + y^J| \\
			&\geq \ell_{2 \max(\Vert x \Vert, \Vert y\Vert)} |x^I-x^J - y^I + y^J|.
		\end{align*}
		Consequently:
		\begin{align*}
			\EE_{x,y}[\Vert X_1 - Y_1 \Vert^2] - \Vert x-y \Vert^2
			&\leq - C \ell_{2 \max(\Vert x \Vert, \Vert y\Vert)} \EE_{x,y}[(x^I - x^J - y^I + y^J)^2] \\
			&= - \frac{C \ell_{2 \max(\Vert x \Vert, \Vert y\Vert)}}{N(N-1)} \sum_{i,j=1}^N (x^i - x^j - y^i + y^j)^2 \\
			&= - \frac{2C \ell_{2 \max(\Vert x \Vert, \Vert y\Vert)}}{N-1} \Vert x-y \Vert^2,
		\end{align*}
		where in the last step we used Lemma \ref{lem:normx}. The desired bound follows.
	\end{proof}

	\begin{lemma}
		\label{lem:contraction_estimate}
		Let $\theta_0>0$ be the constant provided by Lemma \ref{lem:Foster-Lyapunov-drift}. For any $\theta \in (0,\theta_0]$ and $\beta > 0$, define $\Gamma_{\theta,\beta}: \RR^N \times \RR^N \to [0,+\infty)$ by
		\[
		\Gamma_{\theta,\beta}(x,y)
		= \Vert x-y \Vert^2 (1+\beta V_\theta(x) + \beta V_\theta(y)),
		\]
		Let $(X_t,Y_t)$ be the natural coupling. Then, there exist constants $\beta \in (0,1)$ and $\kappa \in (0,1)$, depending on $\theta$, such that for all $x,y \in \Z_N$,
		\[
		\EE_{x,y}[\Gamma_{\theta,\beta}(X_1, Y_1)]
		\leq (1-\kappa) \Gamma_{\theta,\beta}(x,y).
		\]
	\end{lemma}
	
	\begin{proof}
		The idea is to obtain contraction from one of the factors in $\Gamma_{\theta,\beta}(x,y)$, depending on the sizes of $x$ and $y$. When $x$ and $y$ are relatively small, we will obtain contraction of the factor $\Vert x-y\Vert^2$, by means of applying Lemma \ref{lem:contraction_natural_coupling_compact} on a compact set. When $x$ and $y$ are both large, Lemma \ref{lem:Foster-Lyapunov-drift} will provide contraction for the factor $(1+\beta V_\theta(x)+\beta V_\theta(y))$. The delicate case is when either $x$ or $y$ is small and the other is large, requiring a more careful choice of the constants.
		
		Given $\theta \in (0,\theta_0]$, let $\lambda\in(0,1)$ and $R_1>0$ be the constants provided by Lemma \ref{lem:Foster-Lyapunov-drift}. Recall that $\Vert X_1 - x \Vert$ and $\Vert Y_1 - y \Vert$ are a.s.\ bounded by $M:=\sqrt{8}K$. Let us fix $A > 0$ large enough such that $(1-\lambda) e^{\theta A } > e^{\theta M}$; the reason for this choice will become apparent below. Without loss of generality, assume $\Vert x \Vert \geq \Vert y \Vert$. We divide the proof into three cases.
		
		\begin{itemize}
			\item Case 1: $\Vert x \Vert \leq R_1 + A$. Then $V_\theta(X_1)$ and $V_\theta(Y_1)$ are bounded above by some constant $\bar{V}$. Thanks to Lemma \ref{lem:contraction_natural_coupling_compact}, we have:
			\begin{align*}
				\EE_{x,y}[\Gamma_{\theta,\beta}(X_1, Y_1)]
				&\leq (1+2\beta \bar{V}) \EE_{x,y}[\Vert X_1 - Y_1 \Vert^2] \\
				&\leq (1+2\beta \bar{V}) (1-\omega(x,y)) \Vert x - y \Vert^2 \\
				&\leq (1+2\beta \bar{V}) (1-\bar{\omega}) \Vert x - y \Vert^2,
			\end{align*}
			where $\bar{\omega} := \frac{4K(1-KL)}{N} \ell_{2R_1+2A}>0$. Choosing $\beta<1$ small enough, the last expression can be bounded above by $(1-\bar{\omega}/2) \Vert x - y \Vert^2$.
			
			\item Case 2: $\Vert x \Vert > R_1 + A$, $\Vert y \Vert > R_1$. Then, thanks to Lemma \ref{lem:Foster-Lyapunov-drift} and the fact that the coupling is a.s.\ non-increasing, we have
			\begin{align*}
				\EE_{x,y}[\Gamma_{\theta,\beta}(X_1, Y_1)]
				&\leq \Vert x - y \Vert^2 \EE_{x,y}[1 + \beta V_\theta(X_1) + \beta V_\theta(Y_1) ] \\
				&\leq \Vert x - y \Vert^2 (1 + \beta \lambda V_\theta(x) + \beta \lambda V_\theta(y) ).
			\end{align*}
			Thus:
			\[
			\frac{\EE_{x,y}[\Gamma_{\theta,\beta}(X_1, Y_1)]}{\Gamma_{\theta,\beta}(x,y)}
			\leq \lambda + \frac{1-\lambda}{1+\beta V_\theta(x)+\beta V_\theta(y)}
			\leq \lambda + \frac{1-\lambda}{1+2\beta}
			=: \tilde{\lambda}
			< 1.
			\]
			
			\item Case 3: $\Vert x \Vert > R_1 + A$, $\Vert y \Vert \leq R_1$. Since $\Vert Y_1 \Vert \leq \Vert y \Vert + M$, as in Case 2, we have
			\begin{align*}
				&\EE_{x,y}[\Gamma_{\theta,\beta}(X_1, Y_1)] \\
				&\leq \Vert x - y \Vert^2 (1 + \beta \lambda V_\theta(x) + \beta e^{\theta M} V_\theta(y) ) \\
				&\leq \Vert x - y \Vert^2 \left(
				\lambda \{ 1 + \beta V_\theta(x) + \beta V_\theta(y) \}
				+ 1-\lambda +\beta e^{\theta M} V_\theta(y) - \lambda \beta V_\theta(y)
				\right).
			\end{align*}
			Thus:
			\begin{align*}
				\frac{\EE_{x,y}[\Gamma_{\theta,\beta}(X_1, Y_1)]}{\Gamma_{\theta,\beta}(x,y)}
				&\leq \lambda + \frac{1-\lambda + \beta e^{\theta M} V_\theta(y) - \beta \lambda V_\theta(y)}{1+\beta V_\theta(x) + \beta V_\theta(y)} \\
				&\leq \lambda + \frac{1-\lambda + \beta e^{\theta M} V_\theta(y)}{1 + \beta e^{\theta A} V_\theta(y)},
			\end{align*}
			because $V_\theta(x) \geq e^{\theta (R_1 + A)} \geq e^{\theta A} V_\theta(y)$. Now, thanks to our previous choice of $A$, it can be easily seen that the last expression is decreasing as a function of $V_\theta(y)$. Since $V_\theta(y) \geq 1$, we obtain:
			\[
			\frac{\EE_{x,y}[\Gamma_{\theta,\beta}(X_1, Y_1)]}{\Gamma_{\theta,\beta}(x,y)}
			\leq \lambda + \frac{1-\lambda + \beta e^{\theta M}}{1+\beta e^{\theta A}}
			=: \hat{\lambda}
			< 1.
			\]
		\end{itemize}
		
		Taking $\kappa = \min(\bar{\omega}/2, 1-\tilde{\lambda},1-\hat{\lambda})$, the result follows.
	\end{proof}

	We are ready to prove Theorem \ref{thm:main:geometric_convergence}:
	
	\begin{proof}[Proof of Theorem \ref{thm:main:geometric_convergence}]
		The constant $\theta_0$ is provided by Lemma \ref{lem:Foster-Lyapunov-drift}. Given any $\theta \in (0,\theta_0]$, let $\beta, \kappa \in (0,1)$ be the constants provided by Lemma \ref{lem:contraction_estimate} (they depend on $\theta$), and recall that we denote $\Gamma_{\theta,\beta}(x,y) = \Vert x - y \Vert^2 (1+\beta e^{\theta \Vert x \Vert } + \beta e^{\theta \Vert y \Vert })$.
		
		For $Y_0 \sim \pi$ independent of $X_0$, let $(X_t,Y_t)$ be the natural coupling; thus, $Y_t \sim \pi$ for all $t\geq 0$. It suffices to bound $q_t := \EE[\Gamma_{\theta,\beta}(X_t,Y_t)] \geq W_2^2( \Law(X_t), \pi)$. Iterating the inequality provided by Lemma \ref{lem:contraction_estimate}, yields $q_t \leq q_0 (1-\kappa)^t$. It remains to bound $q_0$. Let
		\[
		\tilde{C} = \EE[(1+\Vert Y_0 \Vert^2) e^{\theta_0 \Vert Y_0 \Vert}],
		\]
		which is finite thanks to Theorem \ref{thm:stationary}. Thus:
		\begin{align*}
			q_0
			&= \EE\left[\Vert X_0 - Y_0\Vert^2 (1 + \beta e^{\theta \Vert X_0 \Vert } + \beta e^{\theta \Vert Y_0 \Vert })\right] \\
			&\leq 2\EE\left[
			\Vert X_0 \Vert^2
			+ \Vert Y_0\Vert^2
			+ \Vert X_0 \Vert^2 e^{\theta \Vert X_0 \Vert }
			+ \Vert X_0 \Vert^2 e^{\theta \Vert Y_0 \Vert }
			+ \Vert Y_0 \Vert^2 e^{\theta \Vert X_0 \Vert }
			+ \Vert Y_0 \Vert^2 e^{\theta \Vert Y_0 \Vert }
			\right] \\
			&\leq 2 \EE\left[
			\tilde{C} \Vert X_0 \Vert^2
			+ \tilde{C}
			+ \tilde{C} \Vert X_0 \Vert^2 e^{\theta \Vert X_0 \Vert }
			+ \tilde{C} \Vert X_0 \Vert^2
			+ \tilde{C} e^{\theta \Vert X_0 \Vert }
			+ \tilde{C}
			\right] \\
			&\leq 4 \tilde{C} \EE\left[ (1+\Vert X_0 \Vert^2) (1 + e^{\theta \Vert X_0 \Vert} ) \right] \\
			&\leq C \EE\left[ (1+\Vert X_0 \Vert^2) e^{\theta \Vert X_0 \Vert} \right],
		\end{align*}
		for $C = 8\tilde{C}$. The result is proven.
	\end{proof}

	\section{Properties of the Elo stationary distribution}
	\label{sec:properties_stationary}

	We now study the properties of the stationary distribution $\pi$. As mentioned earlier, the ratings are biased estimators of the true skills. However, it is straightforward to verify that the \emph{average predicted score} is unbiased. Although this is already mentioned in the literature, for the reader's convenience, we provide a precise statement and proof in the following proposition.
	
	We will use the symbol $\equald$ for ``equality in the sense of distribution''. Also, we adopt the notation of Remark \ref{rmk:xIJ}: denote $(I,J)$ a random pair of distinct indices chosen uniformly at random, and, conditionally on $I=i, J=j$, the variable $S^{IJ}$ is an independent copy of $S^{ij}$.
	
	\begin{proposition}
		\label{prop:EbXiXj}
		Let $X \sim \pi$. For any $i\in\{1,\ldots,N\}$:
		\[
		\EE \left[ \sum_{\substack{j=1\\ j\neq i}}^N b(X^i - X^j) \right]
		= \sum_{\substack{j=1\\ j\neq i}}^N b(\rho^i - \rho^j).
		\]
	\end{proposition}
	
	\begin{proof}
		Since $\pi$ is stationary, from \eqref{eq:xIJ} we have
		\[
		X
		\equald
		X + K\{ S^{IJ} - b(X^I - X^J)\} (e_I - e_J).
		\]
		Taking expectation (component-wise on $\RR^N$) and cancelling $\EE[X]$, yields
		\begin{align*}
			0
			&= \EE \left[ \sum_{\substack{i,j=1\\ i\neq j}}^N K \{ S^{ij} - b(X^i - X^j) \} (e_i - e_j) \right] \\
			&= K \sum_{\substack{i,j=1\\ i\neq j}}^N \{ b(\rho^i - \rho^j) - \EE[b(X^i - X^j)] \} (e_i - e_j),
		\end{align*}
		where we have used that $\EE[S^{ij}] = b(\rho^i - \rho^j)$. We have $b(\rho^i - \rho^j) = -b(\rho^j - \rho^i)$ thanks to $b(\cdot)$ being odd, and similarly for $b(X^i - X^j)$. Dividing by $K$ and extracting the $i$-th component, the result follows.
	\end{proof}
	
	Since the ratings are biased estimators of the true skills, an obvious challenge is to quantify how far apart is $X$ from $\rho$, especially as $K\to 0$. The following result is a step in this direction.
	
	\begin{lemma}
		\label{lem:EXrhobXrho}
		Let $X\sim \pi$. Then:
		\[
		K \var(S^{IJ})
		\leq \EE\left[ |X^I - X^J - \rho^I + \rho^J| \cdot |b(X^I - X^J) - b(\rho^I - \rho^J)| \right]
		\leq 2K.
		\]
	\end{lemma}
	
	\begin{proof}
		For simplicity, we assume $N=2$, the proof of the general case being a straightforward extension.  Since $X \sim \pi$ is stationary, from \eqref{eq:xIJ}, we see that
		\begin{equation}
			\label{eq:X1X2}
			\begin{pmatrix}
				X^1 \\
				X^2
			\end{pmatrix}
			\equald
			\begin{pmatrix}
				X^1 + K \{S^{12} - b(X^1-X^2) \} \\
				X^2 - K \{S^{12} - b(X^1-X^2) \}
			\end{pmatrix}.
		\end{equation}
		(Since $b(\cdot)$ is odd and $S^{21} \equald -S^{12}$, it does not matter if the sampled pair is $(1,2)$ or $(2,1)$). Call $D=X^1 - X^2$ and $\delta = \rho^1-\rho^2$. Denoting $S=S^{12}$ and subtracting in \eqref{eq:X1X2}, we obtain $D \equald D + 2K \{S - b(D) \}$. Thus:
		\[
		D-\delta
		\equald
		\{D-\delta\} - 2K \{b(D) - b(\delta)\} + 2K \{S - b(\delta)\}
		\]
		Squaring, yields
		\begin{align*}
			\{D-\delta\}^2
			&\equald
			\{D-\delta\}^2 + 4K^2 \{b(D) - b(\delta)\}^2 + 4K^2 \{S - b(\delta)\}^2
			- 4K \{D-\delta\}\{b(D) - b(\delta)\} \\
			&\quad {} + 4K \{D-\delta\}\{S - b(\delta)\}
			- 4K \{b(D) - b(\delta)\}\{S - b(\delta)\}.
		\end{align*}
		We know that $S$ is independent of $D$, that $\EE[S] = b(\delta)$, and $\EE[b(D)] = b(\delta)$, thanks to Proposition \ref{prop:EbXiXj}. Thus, the last two terms have expected value equal to 0. Taking expectation and cancelling $\EE[\{D-\delta\}^2]$ and $4K$, yields
		\[
		\EE\left[ \{D-\delta\} \{b(D) - b(\delta)\} \right]
		= K \var(b(D)) + K \var(S).
		\]
		Since $b(\cdot)$ is increasing, we can add absolute values inside the expectation on the left-hand side. Since $\var(S) \leq \var(b(D)) + \var(S) \leq 2$, the result follows.
	\end{proof}
	
	\begin{remark}
		\label{rmk:lem:EXrhobXrho}
		Lemma \ref{lem:EXrhobXrho} can be interpreted as follows. For any $\gamma>0$, consider the function $g_\gamma(u) := |u-\gamma||b(u)-b(\gamma)|$; it is a measure of how far $u$ is from $\gamma$. Thus, Lemma \ref{lem:EXrhobXrho} says that when comparing a random difference of ratings $X^I - X^J$ against the corresponding difference of true skills $\rho^I - \rho^J$ using this notion, the expected value is of order exactly $K$. Note that, for typical choices of the function $b(\cdot)$ (for instance, the logistic \eqref{eq:logistic}), $g_\gamma(\cdot)$ is approximately quadratic around $\gamma$, and linear away from it. See Section \ref{sec:numerics:expectation_K} for numeric results in this direction.
	\end{remark}
	
	The following result is an application of Lemma \ref{lem:EXrhobXrho}. It provides an estimate of order $\sqrt{K}$ for the expectation of $|X^I - X^J - \rho^I + \rho^J|$:
	
	\begin{lemma}
		\label{lem:EXrho}
		Let $X \sim \pi$. Denote $\eta = 1+2\max_i |\rho^i|$. Then, for all $K \leq \ell_\eta/2$, we have
		\[
		\EE\left[|X^I - X^J - \rho^I + \rho^J|\right]
		\leq \sqrt{\frac{8K}{\ell_\eta}}.
		\]
	\end{lemma}
	
	\begin{proof}
		Again, we assume $N=2$, the proof for the general case is a straightforward extension. As in the proof of Lemma \ref{lem:EXrhobXrho}, denote $D=X^1-X^2$ and $\delta = \rho^1-\rho^2$. Let $\epsilon \leq 1$. Note that if $|u-\delta| \geq \epsilon$, since $b(\cdot)$ is non-decreasing, we have
		\[
		|b(u) - b(\delta)|
		\geq \min\{ b(\delta+\epsilon)-b(\delta), b(\delta)-b(\delta-\epsilon) \}
		\geq \epsilon \ell_\eta ,
		\]
		thanks to Assumption \ref{ass:b}. Consequently:
		\begin{align*}
			\EE[|D-\delta|]
			&= \EE[|D-\delta|\ind_{\{ |D-\delta| < \epsilon \}} ]
			+ \EE[|D-\delta|\ind_{\{ |D-\delta| \geq \epsilon \}} ] \\
			&\leq \epsilon
			+ \frac{1}{\epsilon \ell_\eta} \EE[|D-\delta| |b(D)-b(\delta)| ] \\
			&\leq \epsilon + \frac{2K}{\epsilon \ell_\eta},
		\end{align*}
		where we have used Lemma \ref{lem:EXrhobXrho}. Choosing $\epsilon = \sqrt{2K/\ell_\eta}$, we have $\epsilon \leq 1$ when $K \leq \ell_\eta/2$, and the result follows.
	\end{proof}

	As a consequence, we obtain the following estimate of order $\sqrt{K}$ for $\EE[\frac{1}{N} |X - \rho|_1]$. This is a refined version of Theorem \ref{thm:main:stationary_sqrtK}:
	
	\begin{theorem}[convergence of $\pi$ to the true skills]
		\label{thm:stationary_sqrtK}
		Let $X\sim\pi$. Denote $\eta = 1+2\max_i |\rho^i|$. Then, for all $K\leq \ell_\eta/2$, we have
		\[
		\EE\left[\frac{1}{N} |X - \rho|_1 \right]
		\leq \frac{N-1}{N} \sqrt{\frac{8K}{\ell_\eta}}.
		\]
	\end{theorem}
	
	\begin{proof}
		Since $X - \rho \in \Z_N$, thanks to Lemma \ref{lem:normx}, we have
		\begin{align*}
			N \EE[|X - \rho|_1]
			&\leq \EE\left[\sum_{i,j=1}^N |X^i - X^j - \rho^i + \rho^j| \right] \\
			&= N(N-1) \EE\left[|X^I - X^J - \rho^I + \rho^J|\right].
		\end{align*}
		Applying Lemma \ref{lem:EXrho}, the result follows.
	\end{proof}
	
	The remainder of this section is devoted to showing that $\supp(\pi) = \Z_N$. Since the proof requires several preliminary lemmas, we outline our general strategy first:
	
	\begin{itemize}
		\item For the simpler case of $N=2$ players and binary scores, i.e., scores in $\{-1,1\}$, show that, starting from any non-random point, the rating of one player can attain any open interval.
		
		\item Generalize for any $N\geq 2$, still for binary scores. That is, show that the ratings can attain any subset of $\Z_N$ which is open in the relative topology.
		
		\item Extend this to general scores in $[-1,1]$. Assumption \ref{ass:suppS} about the support of the scores containing $\{-1,1\}$ will be essential.
		
		\item Finally, deduce $\supp(\pi)=\Z_N$ as a consequence.
	\end{itemize}
	
	We start by showing that for $N=2$ players and binary scores, the ratings can become arbitrarily large, both towards $-\infty$ and $+\infty$. To this end, let us introduce the notation
	\[
	\E_-(u) = u + K \{ -1 - b(2u) \},
	\qquad
	\E_+(u) = u + K \{ 1 - b(2u) \}.
	\]
	For $t\in\NN$, denote $\E_-^t = \E_- \circ \cdots \circ \E_-$, so that $\E_-^t(u)$ is the rating of a player that started from $u\in\RR$ and then lost $t$ games in a row against a player that started from $-u$. Similarly for $\E_+^t(u)$.
	
	\begin{lemma}
		\label{lem:elo^t}
		For all $u\in\RR$,
		\[
		\lim_{t \to \infty} \E_-^t(u) = -\infty,
		\qquad
		\lim_{t \to \infty} \E_+^t(u) = \infty.
		\]
	\end{lemma}
	
	\begin{proof}
		Let's prove the assertion for $\E_+$, the other one being analogous. Since $b(\cdot) < 1$, the sequence $(\E_+^t(u))_{t\in\NN}$ is increasing. Assuming that it converges to some finite value $u_\infty$, we have
		\[
		\E_+^{t}(u) - \E_+^{t-1}(u)
		= K \{1-b(2 \E_+^{t-1}(u))\}
		\geq K \{ 1-b(2u_\infty) \}
		>0.
		\]
		Thus, $\E_+^{t}(u) \geq u + t K \{ 1-b(2u_\infty) \}$, then $\lim_t \E_+^{t}(u) = \infty$, which is a contradiction.
	\end{proof}

	\begin{lemma}
		\label{lem:expansivity-of-inverse-elo-maps}
		For any non-empty interval $\I$ there exists $\eta = \eta(\I) > 0$ such that for all $u,v \in \I$
		\[
		\min \{ |\E_-^{-1}(u) - \E_-^{-1}(v)|, |\E_+^{-1}(u) - \E_+^{-1}(v)| \}
		\geq (1+\eta) |u-v|.
		\]
	\end{lemma}
	
	\begin{proof}
		We first prove the lower bound for $\E_+$, the argument for $\E_-$ is analogous. Given $u \geq v$ in $\I$, denote $u_* = \E_+^{-1}(u)$ and $v_* = \E_+^{-1}(v)$. We need to show that $|u-v| \leq \lambda |u_*-v_*|$ for some $\lambda = \lambda(\I) < 1$. Write $\I = (\ba,\bb)$. Note that $\E_+^{-1}$ generates a jump of size at most $2K$, which implies that $u_*, v_* \in (\ba-2K, \bb+2K)$. Set $M = 2 \max\{|\ba-2K|, |\bb+2K|\}$. Note that the quantity
		\begin{align*}
			u-v
			&= u_* + K \{ 1-b(2 u_*) \} - v_* - K \{ 1 - b(2v_*) \} \\
			&= u_* - v_* - K \{ b(2 u_*) - b(2 v_*) \}
		\end{align*}
		lies in the interval $[(1-2KL)(u_* - v_*), (1-2 K \ell_M)(u_* - v_*)]$, because
		\[
		2 \ell_M (u_* - v_*)
		\leq b(2 u_*) - b(2 v_*)
		\leq 2L (u_* - v_*).
		\]
		Consequently, for $\lambda = \max\{|1-2KL|, |1-2 K \ell_M| \}$, we obtain $|u-v| \leq \lambda |u_*-v_*|$. Since $1-2KL>-1$, we have $\lambda < 1$, which concludes the proof.
	\end{proof}

	\begin{lemma}
		\label{lem:path_ab}
		Let $\ba < \bb$ and $u \in \RR$. Then there exist $t\in \NN$ and a path $\tau = \E_t \circ \cdots \circ \E_1$, where each $\E_k$ is either  $\E_-$ or $\E_+$, such that $\tau(u) \in (\ba,\bb)$.
	\end{lemma}
	
	\begin{proof}
		First assume $\bb-\ba \geq 2K$. Thanks to Lemma \ref{lem:elo^t}, we know that either the sequence $(\E_-^t(u))_{t\in\NN}$ or $(\E_+^t(u)))_{t\in\NN}$ will eventually hit $(\ba,\bb)$, because the jump size is smaller than $2K$.
		
		We now assume $\bb-\ba<2K$. It suffices to find a path $\gamma$ such that the length of the interval $\gamma^{-1}((\ba,\bb))$ is greater than $2K$. We will generate $\gamma^{-1} = \E_{t_*}^{-1} \circ \cdots \circ \E_1^{-1}$ as follows. Set $\I = [\ba-2K, \bb+2K]$, $\I_0 = (\ba,\bb)$ and $t=1$. Then:
		\begin{itemize}
			
			\item If $\E_+^{-1}(\I_{t-1}) \subseteq \I$, set $\E_t = \E_+$; otherwise $\E_t = \E_-$.
			
			\item Set $\I_t = \E_t^{-1}(\I_{t-1})$.
			
			\item If $\I_t$ has length $2K$ or more, stop.
			
			\item Otherwise, update $t \leftarrow t+1$ and return to the first step.
		\end{itemize}
		Note that both $\E_-^{-1}$ and $\E_+^{-1}$ generate a jump of size at most $2K$, and since the length of $\I$ is bigger than $4K$, then $\I_t \subseteq \I$ at every step of the procedure. Consequently, the length of $\I_t$ is at least $(1+\eta)^t (\bb-\ba)$ thanks to Lemma \ref{lem:expansivity-of-inverse-elo-maps}, where $\eta = \eta(\I) > 0$ is the constant provided by the lemma. This guarantees that eventually the procedure stops at some time $t_*$, with the final interval $\I_{t_*}$ having length $2K$ or more.
	\end{proof}
	
	We now extend the previous lemma to any $N \geq 2$. Recall the notation \eqref{eq:E_alpha}: for $x\in\Z_N$, distinct indices $i,j$ and non-random $s \in [-1,1]$, denote $\E_\alpha(x) \in \Z_N$ the vector obtained after a single step of the Elo algorithm using the information in $\alpha = (i,j,s)$.
	
	\begin{lemma}
		\label{lem:path_U}
		Fix $N\geq 2$ and $x \in \Z_N$. Then, for any non-empty set $U \subseteq \Z_N$ open in the relative topology of $\Z_N$, there exist $t\in\NN$ and a path $\tau = \E_{\alpha_t} \circ \cdots \circ \E_{\alpha_1}$ satisfying $\tau(x) \in U$, where each $\alpha_k = (i_k, j_k, s_k)$ is such that $s_k \in \{-1,1\}$.
	\end{lemma}
	
	\begin{proof}
		Given a collection of open intervals $(\ba^1, \bb^1), \ldots, (\ba^{N-1}, \bb^{N-1})$, we must show that there exists a path $\tau$ such that for each $i \leq N-1$, the $i$-th component of $\tau(x)$ belongs to $(\ba^i, \bb^i)$. The idea is to apply Lemma \ref{lem:path_ab} using the $N$-th player as a pivot, in order to bring the rating of every other player to the target interval.
		
		Start with player $i=1$. Denote $m = (x^1+x^N)/2$, set $u = x^1 - m$ and $v = x^N - m$, thus $u+v = 0$. Thanks to Lemma \ref{lem:path_ab}, there exists a path (of $\E_-$ and $\E_+$) that brings $u$ to the interval $(\ba^1 - m, \bb^1 - m)$. Since $2u = u - v = x^1 - x^N$, this path has an equivalent path $\tau_1$ of operators $\E_\alpha$ with $\alpha$ of the form $(1,N,s)$ and $s\in\{-1,1\}$, that brings $x^1$ to the interval $(\ba^1,\bb^1)$.
		
		We can now repeat this argument for $x^2$ and the $N$-th component of $\tau_1(x)$, which yields a path $\tau_2$ of games between players $2$ and $N$ that brings $x^2$ to the interval $(\ba^2,\bb^2)$. Repeating this $N-1$ times and setting $\tau = \tau_{N-1} \circ \cdots \circ \tau_1$, the conclusion follows.
	\end{proof}

	The previous lemma is equivalent to saying that, for scores in $\{-1,1\}$, one has $\PP(X_t \in U \mid X_0=x)>0$ for some $t\in\NN$. We now extend this to general scores in $[-1,1]$. Recall that Assumption \ref{ass:suppS} guarantees that for every $\epsilon>0$, one has $\PP(S^{ij} < -1+\epsilon) > 0$ and $\PP(S^{ij} > 1-\epsilon) > 0$, for any distinct $i,j$.

	\begin{lemma}
		\label{lem:PXt_in_U}
		Fix $N\geq 2$ and $x\in\Z_N$. Then, for any non-empty set $U \subseteq \Z_N$ open in the relative topology of $\Z_N$, there exists $t \in \NN$ such that $\PP(X_t \in U \mid X_0 = x)>0$.
	\end{lemma}
	
	\begin{proof}
		By Lemma \ref{lem:path_U}, there exists a path $\tau = \E_{\alpha_t} \circ \cdots \circ \E_{\alpha_1}$ satisfying $\tau(x) \in U$, where each $\alpha_k = (i_k, j_k, s_k)$ is such that $s_k \in \{-1,1\}$. The idea is to \emph{broaden} this sequence of games into a set of paths with positive probability. Since $x\in\Z_N$ is fixed, and fixing the indices $i_1,j_1,\ldots,i_t,j_t$, we can write $\tau(x)$ as a function of $s_1,\ldots,s_t$. Specifically, let $f: [-1,1]^t \to \Z_N$ be the function
		\[
		f(r_1,\ldots,r_t)
		= (\E_{(i_t,j_t,r_t)} \circ \cdots \circ \E_{(i_1,j_1,r_1)})(x),
		\]
		thus $\tau(x) = f(s_1,\ldots,s_t)$. Clearly $f$ is continuous. Thus, since $U$ is open, the pre-image $f^{-1}(U)$ is open in $[-1,1]^t$ and contains the point $(s_1,\ldots,s_t)$. Consequently, there exists $\epsilon>0$ such that the set
		\[
		R_\epsilon = (s_1-\epsilon,s_1+\epsilon) \times \cdots \times (s_t-\epsilon,s_t+\epsilon) \cap [-1,1]^t
		\]
		is contained in $f^{-1}(U)$. Recalling that $I_{k-1},J_{k-1}$ denote the (random) players of game $k$, we thus have
		\[
		\PP(X_t \in U \mid X_0 = x)
		\geq \PP\left( \bigcap_{k=1}^t \left\{ I_{k-1}=i_k, J_{k-1}=j_k, S_{k-1}^{i_k j_k} \in (s_k-\epsilon,s_k+\epsilon) \cap [-1,1] \right\} \right),
		\]
		which is strictly positive thanks to Assumption \ref{ass:suppS}.
	\end{proof}
	
	Finally, we are ready to prove the following result regarding the support of $\pi$:
	
	\begin{theorem}[full support of $\pi$]
		\label{thm:support}
		$\supp(\pi) = \Z_N$.
	\end{theorem}
	
	\begin{proof}
		Let's argue by contradiction: assume there exists a point in $\Z_N \setminus \supp(\pi)$. Since by definition $\supp(\pi)$ is closed, there exists $U \subseteq \Z_N \setminus \supp(\pi)$ non-empty and open. Now, consider the Elo process $(X_t)_{t\in\NN}$ starting with $X_0 \sim \pi$, thus $X_t \sim \pi$ for all $t\in\NN$. Since $\pi(U)=0$, we thus have $\PP(X_t \in U)=0$ for all $t\in\NN$. Consequently:
		\begin{align*}
			0
			&= \sum_{t\in\NN} 2^{-t} \PP(X_t \in U) \\
			&= \sum_{t\in\NN} 2^{-t} \int_{\Z_N} \PP(X_t \in U \mid X_0=x) \pi(dx) \\
			&= \int_{\Z_N} \sum_{t\in\NN} 2^{-t} \PP(X_t \in U \mid X_0=x) \pi(dx).
		\end{align*}
		Thanks to Lemma \ref{lem:PXt_in_U}, the integrand is strictly positive. This implies that $\pi \equiv 0$, which is a contradiction.
	\end{proof}

	\section{Numerical results}
	\label{sec:numerics}
	
	We performed Monte Carlo simulations to numerically examine relevant features of the stationary distribution $\pi$, in order to illustrate some of the mathematical results of Section \ref{sec:properties_stationary}, and acquire new insights. In this section we present the result of these experiments. We considered:
	\begin{itemize}
		\item $N=2$ players. 
		\item $b(x) = \tanh(Lx)$, for $L=0.5$.
		\item Scores in $\{-1,1\}$.
	\end{itemize}
	This leaves two parameters: $K>0$ and $\rho^1 \in \RR$, and we explored the dependence of $\pi$ on both. In our simulations, one of them will be kept fixed while the other varies along a finite set of values. For every $K$ and $\rho^1$, we:
	\begin{itemize}
		\item Fix a large number of samples $m$.
		
		\item Generate $m$ realizations of the initial condition $X_0^1$, by sampling $m$ independent copies of a given initial distribution on $\RR$ (for instance, uniform on some interval).
		
		\item Run the Elo algorithm independently for the $m$ copies, for a number of steps $t_*$ large enough to reach equilibrium ($t_* = 200$ seems to suffice for most cases).
	\end{itemize}
	Afterwards, we can approximate some feature of interest of $\pi$ by a Monte Carlo procedure. For instance, we estimate an expected value as an average, or compute a histogram to visualize the distribution of $X^1 \sim \operatorname{Proj}_1(\pi)$. We performed four different simulations, which we now proceed to describe.
	
	\subsection{Density for different values of \texorpdfstring{$\rho^1$}{rho-1}}
	\label{sec:numerics:density_rho}
	
	For $K=0.4$ and five different values of $\rho^1$ ranging from $0$ to $1$, we computed and plotted a normalized histogram with $m=5 \times 10^7$ samples per curve, which provides a good approximation for the probability density function of $X^1$, assuming it exists. The resulting plots are shown in Figure \ref{fig:rhos}.
	
	\begin{figure}[t!]
		\centering
		\includegraphics[width=0.5\linewidth]{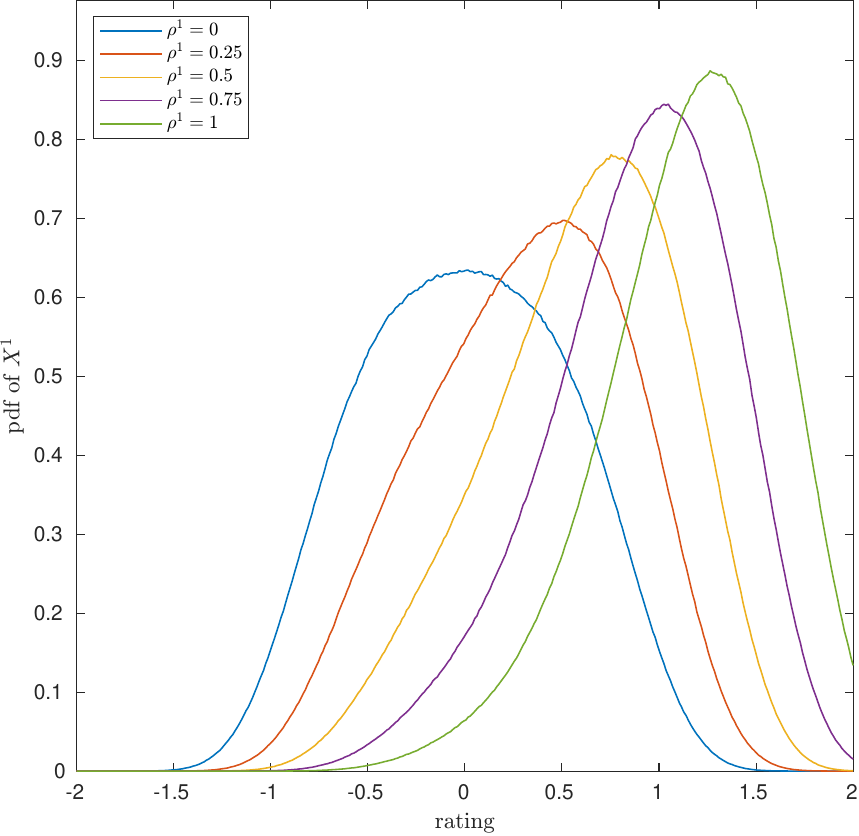}
		\caption{Monte Carlo approximation of the probability density function of $X^1$ for five different values of $\rho^1$. Parameters: $K=0.4$, $L=0.5$. Samples per curve: $m=5 \times 10^7$.} 
		\label{fig:rhos}
	\end{figure}
	
	For these parameters, the curves appear smooth and somewhat bell-shaped. The specific shape does seem to depend on $\rho^1$: the curves are not symmetric about their mean, except of course for $\rho^1 = 0$, and they become narrower as $\rho^1$ increases.

	\subsection{Density for different values of \texorpdfstring{$K$}{K}}
	\label{sec:numerics:density_K}
	
	Similarly, for $\rho^1=0$ and five different values of $K$ ranging from $0.02$ to $1.2$, we computed and plotted a normalized histogram with $m=5 \times 10^7$ samples per curve. The result is shown in Figure \ref{fig:Ks}. Observe that all the chosen values of $K$ satisfy the constrain $KL < 1$.
	
	\begin{figure}[t!]
		\centering
		\includegraphics[width=0.5\linewidth]{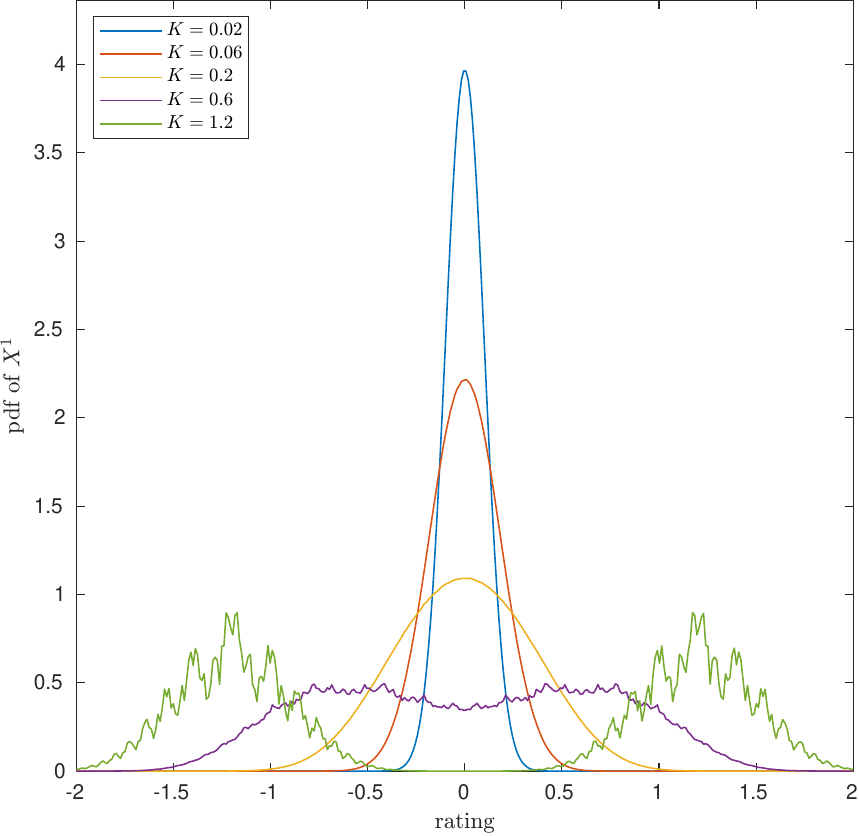}
		\caption{Monte Carlo approximation of the probability density function of $X^1$ for five different values of $K$. Parameters: $\rho^1=0$, $L=0.5$. Samples per curve: $m=5 \times 10^7$.} 
		\label{fig:Ks}
	\end{figure}
	
	Naturally, the curves are symmetric about $\rho^1 = 0$. As predicted by the results of Section \ref{sec:properties_stationary} (for instance Lemma \ref{lem:EXrhobXrho} and Theorem \ref{thm:stationary_sqrtK}), the curve is concentrated around $0$ for $K$ small, and it becomes more spread as $K$ increases. Interestingly, this spread is due to the emergence of two bumps, rather than a single bump becoming wider. Perhaps surprisingly, even though the curves appear to be smooth for $K$ small, they seem to become progressively rougher as $K$ increases. An interesting problem is to study mathematically the existence and smoothness/roughness of the density, as a function of the parameters.

	\subsection{Expected rating vs true skill}
	\label{sec:numerics:expected_rating}
	
	For $K=1$ and 101 linearly spaced values of $\rho^1$ ranging from $-1$ to $1$, we approximated $\EE[X^1]$ by its corresponding average, with $m=5 \times 10^5$ samples per point. The goal is to compare $\EE[X^1]$ with $\rho^1$. To do so, in Figure \ref{fig:expectation_rhos} we plotted $b(2\EE[X^1])$ as a function of $b(2 \rho^1)$. For reference, we also plotted the corresponding average that approximates the quantity $\EE[b(2X^1)]$.
	
	\begin{figure}[t!]
		\centering
		\includegraphics[width=0.5\linewidth]{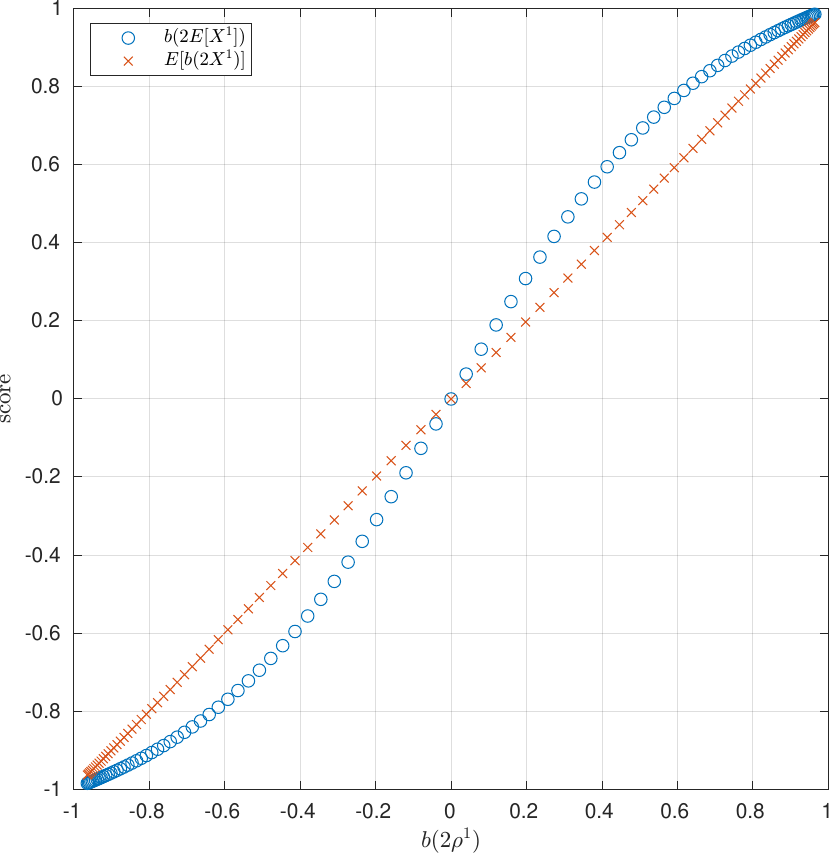}
		\caption{Monte Carlo estimation of $b(2\EE[X^1])$ and $\EE[b(2X^1)]$ for 101 values of $\rho^1$, plotted as functions of $b(2\rho^1)$. Parameters: $K=1$, $L=0.5$. Samples per point: $m=5 \times 10^5$.} 
		\label{fig:expectation_rhos}
	\end{figure}
	
	We observe that the plot of $b(2\EE[X^1])$ differs from the diagonal except at the origin, that is, $\EE[X^1] \neq \rho^1$ for all values of $\rho^1$ except $0$. This means that the rating is a biased estimator of the true skill in general, which is already well-known. Specifically, in this case $X^1$ seems to overestimate $\rho^1$ when it is positive, and underestimate it when it is negative. Let us mention that we chose the relatively large value $K=1$ (comparable to $1/L = 2$) so that the difference would be noticeable in the plot. For smaller values of $K$ the difference is more subtle, although still present; see for instance \cite{manCastillo-junca2024,manCastillo-junca2024b}. On the other hand, the plot of $\EE[b(2X^1)]$ falls right on the diagonal, which of course is to be expected because $b(2X^1)$ is an unbiased estimator of $b(2\rho^1)$ (Proposition \ref{prop:EbXiXj}).
	
	\subsection{Expected distance between rating and true skill vs \texorpdfstring{$K$}{K}}
	\label{sec:numerics:expectation_K}
	
	Similarly, for $\rho^1=0.5$ and 100 logarithmically spaced values of $K$ from $10^{-3}$ to $1$, we approximated $\EE[|X^1-\rho^1|]$ by its corresponding average, with $m=5 \times 10^4$ samples per point. Since some values of $K$ are very small, we needed to run the simulation for much longer in order to reach equilibrium. In Figure \ref{fig:expectation_Ks} we plotted $\EE[|X^1-\rho^1|]$ as a function of $K$, together with the curve $C\sqrt{K}$ for comparison, for $C$ such that both plots agree on the the first point. A log-log scale plot is also provided.
	
	\begin{figure}[t!]
		\centering
		\includegraphics[width=0.45\linewidth]{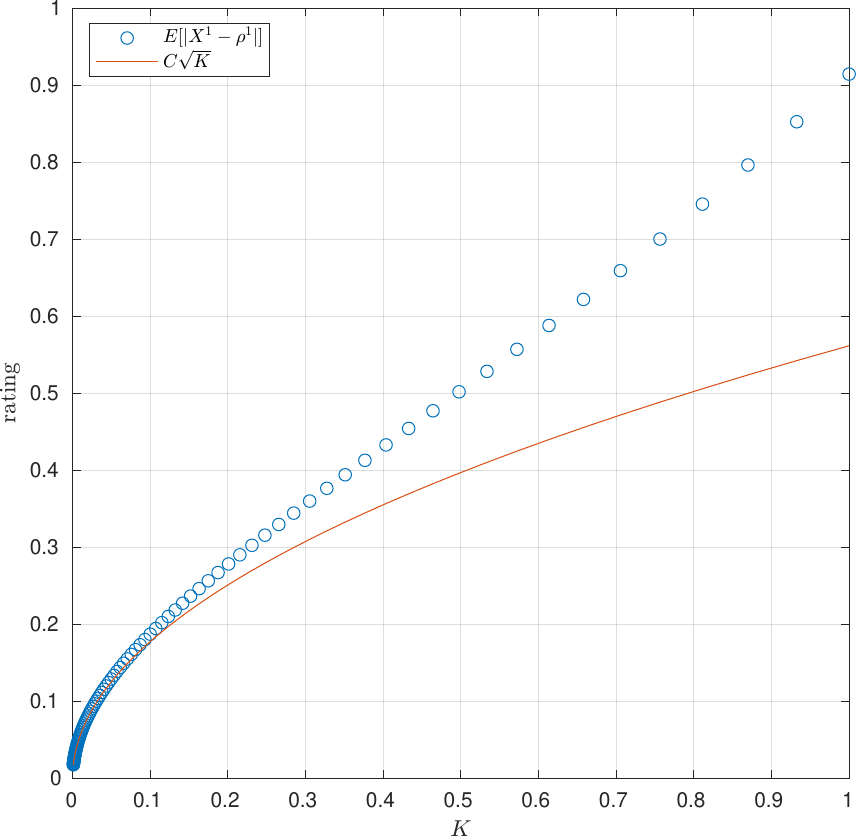}
		\includegraphics[width=0.45\linewidth]{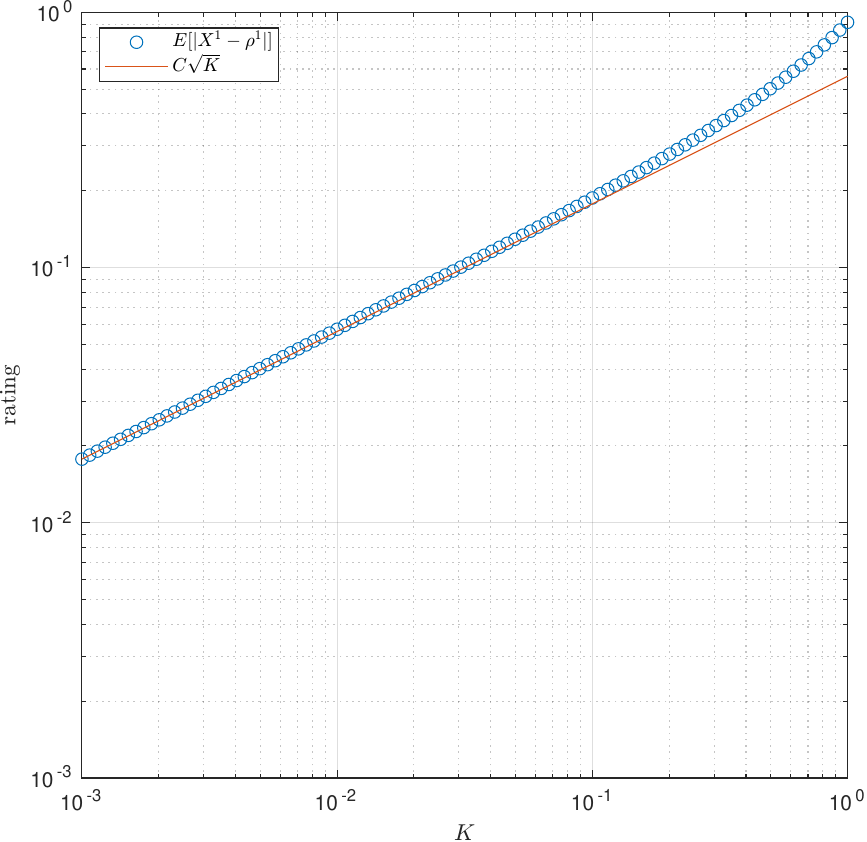}
		\caption{Monte Carlo estimation of $\EE[|X^1-\rho^1|]$ for 100 values of $K$, together with the function $C\sqrt{K}$. Right plot is in log-log scale. Parameters: $\rho^1=0.5$, $L=0.5$. Samples per point: $m=5 \times 10^4$.} 
		\label{fig:expectation_Ks}
	\end{figure}
	
	We observe that for $K$ small, $\EE[|X^1-\rho^1|]$ is approximated very well by the curve $C\sqrt{K}$, in agreement with Theorem \ref{thm:stationary_sqrtK}. This kind of behavior was also observed numerically in \cite{manCastillo-junca2024,manCastillo-junca2024b}, where the standard deviation of $X^1$ is seen to be of order $\sqrt{K}$ as well, for $K\ll 1/L$. For $K$ away from $0$, the increase in $\EE[|X^1-\rho^1|]$ seems to become linear, which suggests that the restriction of $K$ being smaller than some constant in Theorem \ref{thm:stationary_sqrtK} cannot be avoided. Overall, the plot agrees with the lower and upper bounds of Lemma \ref{lem:EXrhobXrho}, see also Remark \ref{rmk:lem:EXrhobXrho}.

	\section{Conclusion and perspectives}
	\label{sec:conclusion}
	
	In this article we studied the Elo process, a Markov chain taking values on the subset of $\RR^N$ of vectors with zero-sum, modeling the evolution of the ratings of $N$ players exchanging points according to the Elo rating algorithm. We provided several results about the large-time behavior of the process and its stationary distribution $\pi$, including exponential convergence to it. We also proved important properties of $\pi$, such as finiteness of an exponential moment, convergence to the Dirac mass at $\rho$ as the $K$-factor vanishes, and full support. We also performed Monte Carlo simulations that illustrate our results and offer new insights. Overall, the present article is a significant contribution towards a deeper understanding of the mathematical properties of Elo rating systems, in this relatively new topic in applied probability.
	
	Finally, let us mention some open questions and possible lines of future research that stem from our work and the relevant literature:
	\begin{itemize}
		\item \emph{Variable strengths.} As mentioned in the Introduction, more realistic variants of the model allow the the vector $\rho$ of true skills or strengths to be updated randomly after each game, as done in \cite{during-fisher-wolfram2022,during-torregrossa-wolfram2019} in a mean-field setting. Extending the results of the present article to the case variable strengths for finite $N$ and fixed $K>0$ would be highly desirable.
		
		\item \emph{Dependence of the constants on $N$}. The constants appearing in some of our results depend on all model parameters; in particular, they deteriorate as $N$ grows. The behavior of the optimal constants as functions of $N$ remains to be investigated. It is likely that some of them are even bounded uniformly in $N$, under suitable hypotheses on the behavior of $X_0$ and the vector $\rho$ as $N$ grows. This possibly includes the constant $\theta_0$ of the bounded exponential moment for $\pi$ in Theorem \ref{thm:stationary}-\ref{thm:stationary:i}, and the value $\kappa \in (0,1)$ giving the exponential rate of convergence in Theorem \ref{thm:main:geometric_convergence} (after properly rescaling time).
		
		\item \emph{Smoothness of $\pi$.} The numerical results of Sections \ref{sec:numerics:density_rho} and \ref{sec:numerics:density_K} suggest that $\pi$ has a density with respect to the Lebesgue measure of $\Z_N$. Even though this is clearly true if the scores have densities in $[-1,1]$, it is not obvious for binary scores. In the latter case, our numerical results indicate that the density, if it exists, appears to be smooth for $K$ small and becomes rougher as $K$ increases. It would be very desirable to have a deeper mathematical understanding of this observed phenomenon.
		
		\item \emph{Computation of the bias.} From simulations, we known that the ratings are biased estimators of the true skills $\rho$, but the actual value of $\EE[X]$ for $X \sim \pi$ remains unknown. Numerical evidence \cite{manCastillo-junca2024,manCastillo-junca2024b} indicates that $\EE[X]$ might not depend on $K$, but it definitely depends on $\rho$. In any case, it would be highly desirable to compute $\EE[X]$ explicitly, even for a specific instance of the scores and function $b(\cdot)$. This would allow one to adjust the ratings and obtain an unbiased estimator.
		
		\item \emph{Mean-field limit.} As mentioned in the Introduction, in \cite{jabin-junca2015} the authors propose a related model, corresponding to the mean-field limit of the Elo process as $N \to \infty$ and $K \to 0$ under a suitable scaling; see also \cite{during-fisher-wolfram2022,during-torregrossa-wolfram2019}. This leads to a kinetic-like partial differential equation modeling the evolution of the distribution of ratings in a large league, as is the case for many popular online games, which evolves according to the Elo rating algorithm with a very small $K$-factor. Alternatively, one could study the mean-field limit only as $N\to\infty$, keeping $K>0$ fixed. A challenging problem is to study this setting mathematically, proving that the mean-field limit is well defined, justifying its validity (propagation of chaos), and showing if and how the properties of the Elo process are reflected in the limit.
	\end{itemize}

	\subsection*{Acknowledgments}
	
	We would like to thank three anonymous referees for their useful suggestions, which allowed us to improve the presentation of the article, and encouraged us to further investigate the properties of the process, ultimately leading to a proof of the exponential convergence result.

	\bibliographystyle{plain}
	\bibliography{references.bib}{}
	
\end{document}